\documentclass[reqno]{amsart}

\usepackage[svgnames, dvipsnames]{xcolor, colortbl}
\usepackage{xspace, xfrac, relsize, csquotes, verbatim}
\usepackage[inline]{enumitem}
\usepackage{amsmath, amssymb, latexsym, thmtools, mathtools, mathrsfs}
\usepackage{caption, graphicx}
\usepackage[alphabetic, msc-links]{amsrefs}
\usepackage{varioref, cleveref}
\hypersetup{colorlinks=true, linkcolor=RoyalBlue, citecolor=RoyalBlue}
\usepackage{a4wide}

\theoremstyle{plain}
\newtheorem{theorem}{Theorem}[section]
\newtheorem{corollary}[theorem]{Corollary}
\newtheorem{lemma}[theorem]{Lemma}
\newtheorem{proposition}[theorem]{Proposition}

\theoremstyle{definition}
\newtheorem{definition}[theorem]{Definition}
\newtheorem{example}[theorem]{Example}

\theoremstyle{remark}
\newtheorem{remark}{Remark}
\newtheorem*{notation}{Notation}

\numberwithin{equation}{section}
\numberwithin{figure}{section}

\DeclareMathOperator{\ch}{\textnormal{ch}}
\DeclareMathOperator{\supp}{\text{supp}}
\newcommand{\g}{\ensuremath{\mathfrak{g}}\xspace}
\newcommand{\h}{\ensuremath{\mathfrak{h}}\xspace}
\newcommand{\W}{\ensuremath{\mathscr{W}}\xspace}
\newcommand{\lw}{\ensuremath{\ell(w)}\xspace}
\newcommand{\ulog}{\ensuremath{-\log U_i(\lambda)}\xspace}

\calclayout
\allowdisplaybreaks
\begin{document}

\title{On tensor products of representations of Lie superalgebras}

\author[A. Das]{Abhishek Das}
\address[A. Das]{Department of mathematics \& statistics\\ indian institute of technology, kanpur, u.p 208016.}
\email{abhidas20@iitk.ac.in}

\author[S.K. Pattanayak]{Santosha Pattanayak}
\address[S.K. Pattanayak]{Department of mathematics \& statistics, indian institute of technology, kanpur, u.p 208016\\ FB 577.}
\email{santosha@iitk.ac.in}
%\urladdr{\url{https://home.iitk.ac.in/~santosha/}}

\date{\today}
    \begin{abstract}
         We consider typical finite dimensional complex irreducible representations of a basic classical simple Lie superalgebra, and give a sufficient condition on when unique factorization of finite tensor products of such representations hold. We also prove unique factorization of tensor products of singly atypical finite dimensional irreducible modules for $\mathfrak{sl}(m+1,n+1)$, $\mathfrak{osp}(2,2n)$, $G(3)$ and $F(4)$ under an additional assumption. This result is a Lie superalgebra analogue of Rajan's fundamental result ~\cite{MR2123935} on unique factorization of tensor products for finite dimensional complex simple Lie algebras. 
    \end{abstract}
%\arrayrulecolor{Blue}
\maketitle

\section{Introduction}\label{secintro}
	\noindent In ~\cite{MR2123935}, Rajan proved that for a complex simple Lie algebra \g, the following is true:
	\begin{theorem}
		Let \(V_1,\dots,V_m\) and \(W_1,\dots,W_n\) be non trivial finite dimensional irreducible representations of \g such that
		\begin{displaymath}
			V_1\otimes\dots\otimes V_m\cong W_1\otimes\dots\otimes W_n.
		\end{displaymath}
		Then \(m=n\), and there is a permutation \(\sigma\) of \(\{1,\dots,n\}\) such that \(V_k\cong W_{\sigma(k)}\) as \g-modules for every \(k\in\{1,\dots,n\}\).
	\end{theorem}
	Proof of the above theorem is based on judicious application of the Weyl character formula and proceeding via inductive arguments upon fixing one of the variables. Authors of ~\cite{MR2980495} extended this result for integrable representations of symmetrizable Kac-Moody algebras up to a one dimensional twist as these algebras admit non-trivial one-dimensional representations. Their proof relies on techniques involving formal logarithm, and the main idea being comparison of an appropriate monomial. 
 
 In ~\cite{MR4504111}, for a complex semisimple Lie algebra \g, the authors consider a more general question of determining all the pairs $(V_1,V_2)$ consisting of two finite dimensional irreducible representations of \g such that $\text{Res}_{\g_0}V_1 \cong \text{Res}_{\g_0}V_2$, where $\g_0$ is the fixed point subalgebra of \g with respect to a finite order automorphism. In ~\cite{MR4343717}, the authors gave a sufficient condition for the uniqueness of tensor products when \g is a Borcherds-Kac-Moody algebra.

 Our goal of this paper is to obtain an analogous result when \g is a basic classical simple Lie superalgebra whose theory is most like that of simple Lie algebras. Nevertheless, there are certain subtle differences that have resulted in many aspects of their representation theory only partly explored. For example, any finite dimensional  representation of $\g$ is completely reducible if and only if $\g=\mathfrak{osp}(1,2n)$. Kac showed that there are two disjoint classes of finite-dimensional irreducible representations of any basic classical Lie superalgebra which he gave appellations as typical and atypical. Typical representations share many properties in common with finite-dimensional representations of simple Lie algebras. In particular, they can be built up explicitly by an induced module construction that further allows a straightforward determination of their characters and dimensions. On the other hand, the situation with atypical representations is far more complicated and they are still not well understood. A serious difficulty being unlike the typical case, here an irreducible  representation is not determined by its central character.~The only class of atypical representations that is by so far tractable is the class of singly atypical representations. So in this note we restrict ourselves to the class of finite dimensional typical and singly atypical representations.

 We first consider typical irreducible finite dimensional representations of \g as they admit a Weyl-Kac character formula.  We give a sufficient condition on when unique factorization of finite
tensor products of such representations hold. We follow techniques involving formal logarithms as developed in ~\cite{MR2980495}. The main strategy is to interpret the given tensor products in terms of products of normalized Weyl numerators. For any typical irreducible finite dimensional representation, we show that, the normalized Weyl numerator factors in accordance with the number of connected components of the set of even simple roots of \g and the factorization doesn't depend on the highest weight that defines the representation. We then apply logarithm on both sides and compare an appropriate monomial that carries all crucial informations of the tensor products. Our result (\autoref{Thm:tnsrpdt}) can be described as follows:

 Let \(V(\lambda_1),\dots,V(\lambda_r),V(\mu_1),\dots,V(\mu_s)\) be typical irreducible finite dimensional representations of a basic classical simple Lie superalgebra \g such that the following isomorphism of \g-modules holds,
	\begin{displaymath}
		V(\lambda_1)\otimes\dots\otimes V(\lambda_r)\cong V(\mu_1)\otimes\dots\otimes V(\mu_s).
	\end{displaymath}
	Then \(r=s\), i.e. the number of tensor constituents on both sides is same. We prove that under some additional hypothesis, each constituent in the left hand side of the above isomorphism is \g-module isomorphic to some constituent in the right hand side, i.e. unique factorization of tensor products holds in this case (\textit{see} \autoref{Thm:tnsrpdt} for details). As a corollary, we show that unique factorization of tensor products of irreducible finite dimensional typical representations holds for superalgebras of types \(A(n,0)\), \(B(0,n)\), \(C_n\), \(F(4)\) and \(G(3)\) (cf. \autoref{unique-coro}). To demonstrate the requirement of the hypothesis in \autoref{Thm:tnsrpdt}, we refer to \autoref{example}.
 
    For $\mathfrak{sl}(m+1,n+1)$, $\mathfrak{osp}(2,2n)$, $G(3)$, and $F(4)$, character formulae for singly atypical finite dimensional irreducible modules are known in the literature (\textit{see} \autoref{sec atyp} for details). We also addressed the same question of unique factorization of tensor products of such modules and found its answer in affirmative under an additional assumption on the highest weights. We state our result (\autoref{Thm:atyp}) below:

    Let \g be any Lie superalgebra among $\mathfrak{sl}(m+1,n+1)$, $\mathfrak{osp}(2,2n)$, $G(3)$ and $F(4)$. Suppose we are given the following isomorphism of \g-modules \[V(\nu_1)\otimes\dots\otimes V(\nu_r)\cong V(\mu_1)\otimes\dots\otimes V(\mu_s),\] where each $\nu_i$'s and $\mu_j$'s are dominant integral singly atypical of type $\gamma$. Then $r=s$, and there is a permutation \(\sigma\) of \(\{1,\dots,r\}\) such that \(V(\nu_k)\cong V(\mu_{\sigma(k)})\) for all $k$.
	
	We briefly outline the contents of each section. \autoref{secprelims} contains preliminaries on Lie superalgebras. In \autoref{lemmas}, we prove some preliminary lemmas which are crucial for the proof of the main theorem. \autoref{secfacto} contains results involving formal logarithm and analysis of normalized Weyl numerators. In \autoref{secmainthm} we provide a proof of our main theorem, and finally in \autoref{sec atyp}, we establish unique decomposition of tensor products for singly atypical finite dimensional irreducible modules for the Lie superalgebras $\mathfrak{sl}(m+1,n+1)$, $\mathfrak{osp}(2,2n)$, $G(3)$ and $F(4)$ under an additional assumption on the weights.
    \begin{notation}
        Throughout this paper we work over the field of complex numbers $\mathbb C$. All modules and algebras are defined over $\mathbb C$ and in addition all the modules are of finite dimension. We write $\mathbb Z_2=\{0,1\}$ and use its standard field structure. 
    \end{notation}
 
	\section{Preliminaries}\label{secprelims}
	 In this section we recall few basic definitions and results pertaining to Lie superalgebras and their representations. We mostly follow the notations of ~\cite{MR3012224} and ~\cite{MR2906817}. To begin with, a superalgebra \(A\) is a \(\mathbb{Z}_2\)-graded vector space \(A=A_0\oplus A_1\) together with a bilinear multiplication satisfying \(A_iA_j\subseteq A_{i+j}\), for \(i,j\in\mathbb{Z}_2\). Degree of a homogeneous element, say \(a\), is denoted by \(|a|\). 
 
\subsection{} A Lie superalgebra \(\mathfrak{g}=\mathfrak{g}_0\oplus\mathfrak{g}_1\) is a superalgebra equipped with a bilinear product \([\cdot\,,\cdot]\) (\emph{bracket}) satisfying the following two axioms: for homogeneous elements \(a,b,c\in\mathfrak{g}\),
	\begin{description}
		\item[Skew-supersymmetry] \([a,b]=-(-1)^{|a||b|}[b,a]\).
		\item[Super Jacobi identity] \([a,[b,c]]=[[a,b],c]+(-1)^{|a||b|}[b,[a,c]]\).
	\end{description}
	The direct summands \(\mathfrak{g}_0\) (resp. $\mathfrak{g}_1$) are called \emph{even} (resp. \emph{odd}) part of \(\mathfrak{g}\), and from the definition it follows that, \(\mathfrak{g}_0\) is a Lie algebra and \(\mathfrak{g}_1\) is a \(\mathfrak{g}_0\)-module.

 \subsection{} A natural analogue of the ordinary simple Lie algebras in the super world are the basic classical Lie superalgebras. In particular, they can be described (with the exception of $A(1, 1) = \mathfrak{psl}(2, 2)$) in terms of a Cartan matrix and generalized root systems. They are defined as follows: 

 A Lie superalgebra $\mathfrak{g}=\mathfrak{g}_0 \oplus \mathfrak{g}_1$ is called basic classical if it satisfies the following conditions: 
 \begin{enumerate}[label=\alph*)]
     \item $\mathfrak g$ is simple,
     \item the Lie algebra $\mathfrak g_0$ is a reductive subalgebra of $\mathfrak g$,
     \item there exists a nondegenerate invariant even supersymmetric bilinear form on \g.
\end{enumerate}
Kac  (\textit{see} ~\cite{MR0519631}, Proposition 1.1) proved  that  the  complete  list  of  basic  classical  Lie  superalgebras, which are not Lie algebras, consists of Lie superalgebras of the type $A(m,n)$, $B(m,n)$, $C(n), D(m,n), F(4), G(3), D(2,1,\alpha)$. We note that the even part of a basic classical Lie superalgebra $\mathfrak{g}$ is a reductive Lie algebra.
	
\subsection{} Following ~\cite{MR3012224}, a Cartan subalgebra \h of \g is defined to be a Cartan subalgebra of the even part \(\g_0\) and the Weyl group \W of \g is simply defined to be the Weyl group of the Lie algebra \(\g_0\). We can choose a non-degenerate even invariant supersymmetric bilinear form $(.|.)$ on \g such that its restriction to  $\h\times\h$ is non-degenerate and \W-invariant. We pull back this non-degenerate bilinear form on \h to get a non-degenerate bilinear form $(.,.)$ on $\h^*$.
	
	Let \h be a Cartan subalgebra of \g. For \(\alpha\in\h^*\), let
	\begin{gather*}
		\g_\alpha\coloneqq\{x\in\g\colon[h,x]=\alpha(h)x,\forall h\in\h\}.
		\intertext{The \textbf{root system} for \g is defined to be}
		\Phi\coloneqq\{\alpha\in\h^*\colon\g_\alpha\not=0,\alpha\not=0\}.
		\intertext{Define the sets of even and odd roots, respectively, to be}
		\Phi_0\coloneqq\{\alpha\in\Phi\colon\g_\alpha\cap\g_0\not=0\},\quad\Phi_1\coloneqq\{\alpha\in\Phi\colon\g_\alpha\cap\g_1\not=0\}.
	\end{gather*}
	For a root $\alpha\in\Phi$ we have $k\alpha\in\Phi$ for an integer $k\ne\pm1$ if and only if $\alpha\in\Phi_1$ and $(\alpha,\alpha)\ne0$; in this case $k=\pm2$. A root $\alpha$ is called \emph{isotropic} if $(\alpha,\alpha)=0$. For any root \(\alpha\), the nondegenerate form on \h gives rise to a unique element \(h_\alpha\in\h\), called the \textbf{coroot} corresponding to \(\alpha\), such that \(\alpha(h)=(h_\alpha,h)\) for every \(h\in\h\). Let \(E\) be the real vector space spanned by \(\Phi\). Then \(\h^*=E\otimes_{\mathbb{R}}\mathbb{C}\) for all basic classical simple Lie superalgebras. We fix a total ordering on \enquote*{$\leqslant$} on \(E\) compatible with the real vector space structure: \(v_1\leqslant w_1\) and \(v_2\leqslant w_2\) imply that \(v_1+v_2\leqslant w_1+w_2\), \(-v_2\leqslant -v_1\), and for any positive real number $c$, \(cv_1\leqslant cv_2\) for all \(v_i,w_j\in E\). We fix such a total order and denote by \(\Phi^+\) (resp. \(\Phi^-\)) the subsets of roots \(\alpha\in\Phi\) such that \(0<\alpha\) (resp. $\alpha<0$). \(\Phi^+\) is called a \textbf{positive system}. The corresponding set of positive even roots is denoted by \(\Phi_0^+\). For further details we refer to Section 1.3.1 of ~\cite{MR3012224}.

  Lie superalgebras of Type I contains a one dimensional center, say \(\text{span}\{z\}\), of \(\g_0\). In these cases \(\h=\h_1\oplus\text{span}\{z\}\), where \(\h_1^*\) is the span of even roots. We extend an element \(\lambda\) in the span of even roots to \(\h^*\) by defining \(\lambda(z)\) to be 0.

	\subsection{} Simple roots are defined in the same way as in the Lie algebra case; but here not all simple systems are conjugate under \W action due to presence of odd roots. Any element of \W is a product of simple reflections corresponding to simple even roots. A simple system containing least number of isotropic roots is called \emph{distinguished} or \emph{standard}. Dynkin diagram corresponding to the standard simple system is called standard Dynkin diagram. For each basic classical simple Lie superalgebra, we can choose a distinguished system of simple roots containing only one isotropic root; this is possible for any basic classical Lie superalgebra except $B(0,n)$, which has no isotropic roots (\textit{see} ~\cite{MR0486011}).
 So, for each basic classical simple Lie superalgebra, we fix such a standard simple system \(\Pi\) and let \(\Pi_0\), \(\Pi_1\) denote the set of even (resp. odd) simple roots in \(\Pi\).
	
\subsection{} We have a triangular decomposition 
 $\g = \mathfrak n^- \oplus \h \oplus \mathfrak n^+$, where $\mathfrak n^{\pm}=\oplus_{\alpha \in \Phi^{\pm}} \g_{\alpha}$. Let's recall some basic facts about the representation theory of basic classical simple Lie superalgebras. Let \g be any such Lie superalgebra. For any \(\lambda\in\h^*\), there is an irreducible (unique up to isomorphism) highest weight \g-module of highest weight \(\lambda\). We denote this module by \(V(\lambda)\). The weight \(\lambda\) is called \emph{dominant integral} if $V(\lambda)$ is finite dimensional (\textit{see} ~\cite{MR3751124}*{page 141}). It is well known that any finite dimensional irreducible representation of \g is of the form \(V(\lambda)\) for some dominant integral weight \(\lambda\). Note that if $\lambda \in \h^*$ is dominant integral, then necessarily we have $\smash{\frac{2(\lambda,\alpha)}{\phantom{2}(\alpha,\alpha)}}\in \mathbb Z_{\geqslant 0}$ for all $\alpha \in \Pi_0$ ~\cite{MR0519631}*{Proposition 2.3}; and we denote this integer by \(\langle\lambda,\alpha\rangle\).

 \begin{remark}{\label{rmkdominant}}
    If \g is of Type I, i.e of type \(A(m,n)\) and \(C_n\), then for a weight \(\lambda\in\h^*\), the condition that \(\langle\lambda,\alpha\rangle\in\mathbb{Z}_{\geqslant 0}\) is also sufficient for being dominant integral (\textit{see} ~\cite{MR3751124} page 132). 
 \end{remark}

 \subsection{} The Weyl vector \(\rho\in\h^*\) is defined by:
	\begin{gather}
		\rho\coloneqq\frac{1}{2}\sum_{\alpha\in\Phi^+_0}\alpha-\frac{1}{2}\sum_{\alpha\in\Phi^+_1}\alpha,\notag
	\intertext{and for any positive simple root \(\beta\), it satisfies (\textit{see} ~\cite{MR3012224}, Proposition 1.33)}
		(\rho,\beta)=\tfrac{1}{2}(\beta,\beta).\label{weylvec}
	\end{gather}
	In particular, $(\rho,\beta)=0$ if \(\beta\) is isotropic.
 
 \noindent A weight \(\lambda\in\h^*\) is said to be \textbf{typical} if \((\lambda+\rho,\alpha)\not=0\) for all isotropic roots \(\alpha\in\Phi_1^+\), and it is called \textbf{atypical} otherwise. A representation associated to a typical weight is called a typical representation. 
 
 \subsection{} Let \(V(\lambda)\) be the finite dimensional irreducible highest weight \g-module of highest weight \(\lambda\). It admits a weight space decomposition: $V(\lambda)=\oplus_{\mu\in\h^*}V(\lambda)_\mu$, where \(V(\lambda)_\mu\) is the weight space corresponding to the weight \(\mu\). The formal character of \(V(\lambda)\) is defined by:
	\begin{align*}
		\ch V(\lambda)&\coloneqq\sum_{\mu\in\h^*}\dim V(\lambda)_\mu e^\mu.
	\end{align*}
	A weight space \(V(\lambda)_\mu\) is zero unless \(\mu=\lambda-\sum_{\alpha\in\Phi^+}n_\alpha\alpha,\ n_\alpha\in\mathbb{Z}_{\geqslant0}\) ~\cite{MR3012224}*{Section 1.5.3}. We note that the finite dimensional irreducible representations of \g are completely determined by their characters (cf. ~\cite{MR2776360}*{Proposition 4.2}).

	\subsection{} If \(\lambda\) is a typical dominant weight, then we have the Weyl Kac character formula (\textit{see} ~\cite{MR0519631}) for \(V(\lambda)\) given by:
	\begin{gather}
		\ch V(\lambda)=\frac{D_1}{D_2}\sum_{w\in\W}(-1)^{\ell(w)}e^{w(\lambda+\rho)},
	\end{gather}
	where \(D_1=\prod_{\alpha\in\Phi^+_1}(e^{\alpha/2}+e^{-\alpha/2})\) and \(D_2=\prod_{\alpha\in\Phi^+_0}(e^{\alpha/2}-e^{-\alpha/2})\). By definition of \(\rho\), the expression \(D_1/D_2\) can also be written as:
	\begin{gather}{\label{d1d2}}
		\frac{D_1}{D_2}=e^{-\rho}\frac{\prod_{\alpha\in\Phi^+_1}(1+e^{-\alpha})}{\prod_{\alpha\in\Phi^+_0}(1-e^{-\alpha})}.
	\end{gather}
	
	We now define the normalized character \(\chi_\lambda\) and the normalized Weyl numerator \(U(\lambda)\) of \(V(\lambda)\) respectively by:
	\begin{gather}
		\chi_\lambda\coloneqq e^{-\lambda}\ch V(\lambda),\\
		U(\lambda)\coloneqq e^{-(\lambda+\rho)}\sum_{w\in\mathscr{W}}(-1)^{\ell(w)}e^{w(\lambda+\rho)}.
		\intertext{By Weyl Kac character formula,}
		D\cdot U(\lambda)=e^{-\lambda}\ch V(\lambda)=\chi_\lambda\label{eqnnc},
	\end{gather}
	where \(D=e^\rho D_1/D_2\). (cf. Equation \eqref{d1d2}).

 \section{Preparatory Lemmas}\label{lemmas}
In this section we prove some preliminary lemmas which are needed for the proof of the main theorem.

 \begin{lemma}{\label{sopr}}
		Let \(\lambda\) be a typical dominant integral weight and \(w\in\W\) be arbitrary. Then \(\lambda+\rho-w(\lambda+\rho)\) is a sum of positive even roots with non negative integral coefficients.
	\end{lemma}
	\begin{proof}
		Put \(\lambda+\rho=\eta\). If \(\lw=1\), say \(w=s_\alpha\), the simple reflection corresponding to the positive even simple root \(\alpha\); then \(\eta-s_\alpha\eta=\langle\eta,\alpha\rangle\alpha\). As \(\lambda\) is dominant, Equation \eqref{weylvec} shows that $\langle\eta,\alpha\rangle$ is positive. Consequently the claim is true in this case. Take a reduced expression of \(w\) and write \(w=us_\beta\text{ with }\ell(u)=\lw-1\). This implies \(u(\beta)\) is positive even root as the set of even roots is \W-invariant. Therefore by induction on \lw, $\eta-w\eta=\eta-us_\beta\eta=\eta-u\eta+\langle\eta,\beta\rangle u(\beta)$ is also a sum of positive even roots with non negative integral coefficients.
	\end{proof}
	
	%Let \(\Pi\) be the standard simple system of \g, and \(\Pi_0\) the subset of even roots in \(\Pi\).
 
 For a typical dominant integral weight \(\lambda\) of \g, we put \(a_\alpha\coloneqq\langle\lambda+\rho,\alpha\rangle\) for each even simple root \(\alpha\). Equation \eqref{weylvec} shows that $a_\alpha$ is a positive integer. By \autoref{sopr}, for $w \in \W$, we can write
	\[\lambda+\rho-w(\lambda+\rho)=\sum_{\alpha\in\Pi_0}c_\alpha(w)\alpha,\quad\text{where each } 
         c_\alpha(w)\in \mathbb{Z}_{\geqslant 0}.\]  
    We set $X(w,\lambda)\coloneqq\prod_{\alpha\in\Pi_0}X_{\alpha}^{c_\alpha(w)}=e^{w(\lambda+\rho)-(\lambda+\rho)}$. Then we have        \begin{equation}\label{numer}
                    U(\lambda)=\sum_{w\in\W}(-1)^{\lw} X(w,\lambda)
              \end{equation}
              
       \subsection{}\label{graph} We now recall few definitions from ~\cite{MR2980495}. The underlying graph \(\mathcal{G}\) of \g is defined to be the graph with vertex set \(\Pi\): two vertices \(\alpha\) and \(\beta\) are joined by an edge iff \((\alpha,\beta)\not=0\). For any subset \(C\) of the vertex set, the subgraph spanned by \(C\) is just the graph having \(C\) as the vertex set. A nonempty subset \(K\subset\Pi_0\) is called \emph{totally disconnected} if it comprises of simple roots that are all mutually orthogonal: \((\alpha,\beta)=0\) for every distinct \(\alpha,\beta\in K\). For any \(w\in\W\), we take a reduced expression, say \(\tilde{w}\) of \(w\). Let \(I(w)\) be the set defined by \(I(w)\coloneqq\{\alpha\in\Pi_0\colon s_\alpha\text{ appears in }\tilde{w}\}\). This is a well defined subset of \(\Pi_0\), (\textit{see} ~\cite{MR1066460}). Let \(\mathcal{I}\coloneqq\{1\not=w\in\W\colon I(w)\text{ is totally disconnected}\}\). Given a totally disconnected subset \(K\in\Pi\), there is a unique element \(w(K)\in\mathcal{I}\) such that $I(w(K))=K$; $w(K)$ is precisely the product of the commuting simple reflections $\{s_\alpha\colon\alpha\in K\}$. This establishes a natural bijection between \(\mathcal{I}\) and the set of all totally disconnected subsets of \(\Pi_0\). Proof of the following lemma follows the same line of argument as in ~\cite{MR2980495}*{Lemma 2}.
        
	\begin{lemma}{\label{keylem}}
		For \(w\in\mathscr{W}\) we have
		\begin{enumerate}[label=(\alph*)]
			\item $I(w)=\{\alpha\in\Pi_0\colon c_\alpha(w)\not=0\}$, i.e $X(w,\lambda)=\prod_{\alpha\in I(w)}X_\alpha^{c_\alpha(w)}$.\label{iw}
			\item For every \(\alpha\in I(w), c_\alpha(w)\geqslant a_\alpha\).\label{c alphgreat a alph}
			\item If \(w\in\mathcal{I}\), then \(c_\alpha(w)=a_\alpha\) for every \(\alpha\in I(w)\).\label{c alph eq a alph}
			\item If \(w\) is not in \(\mathcal{I}\cup\{1\}\), then there exists \(\beta\in I(w)\) such that \(c_\beta>a_\beta\).\label{c bet gret a bet}
		\end{enumerate}
	\end{lemma}
	We recall the following definition from ~\cite{MR2980495} which will be used in the next section.
 
	\begin{definition}{\label{def k partition}}
		Let \(k\) be a positive integer. A \(k\)-partition \(\mathcal{J}\) of the graph \(\mathcal{G}\) is an ordered \(k\)-tuple $\mathcal{J}\coloneqq(J_1,\dots,J_k)$ such that each \(J_i\) is a nonempty totally disconnected subset of the vertex set \(\Pi\); \(J_i\cap J_j=\emptyset\) for \(i\not=j\), and lastly; \(\bigcup_{i=1}^k J_i=\Pi\). For each such partition, we define \(w(\mathcal{J})\coloneqq w(J_1)\cdots w(J_k)\in\W\).
	\end{definition}
	Denote by \(P_k(\mathcal{G})\) the set of all \(k\)-partitions of \(\mathcal{G}\) and put \(c_k(\mathcal{G})\coloneqq|P_k(\mathcal{G})|\), the cardinality of $P_k(\mathcal{G})$. We also denote by $k(\mathcal{G})$ the number \((-1)^{|\mathcal{G}|}\sum_{k=1}^{|\mathcal{G}|}(-1)^k\frac{c_k(\mathcal{G})}{k}\).
	Let the symbol \(X_\alpha\) denote $e^{-\alpha}$ for $\alpha\in\Pi_0$, and consider the algebra of formal power series \(\mathcal{A}\coloneqq\mathbb{C}[[X_\alpha\colon\alpha\in\Pi_0]]\). \autoref{sopr} shows that \(U(\lambda)\in\mathcal{A}\) with constant term 1 corresponding to \(w=1\). The formal logarithm for any element \(\zeta\in\mathcal{A}\) with constant term 1 is defined by:\[\log\zeta\coloneqq-\sum_{k\geqslant 1}(1-\zeta)^k/k.\]
	The following lemma will be used in the proof of the main theorem.
	\begin{lemma}{\label{mainlem}}
		Let \(\mathfrak{g}\) be a basic classical simple Lie superalgebra and \(\lambda,\mu\in\mathfrak{h}^*\) be typical dominant integral weights. Then the following statements are equivalent.\\
		\begin{enumerate*}[label=(\alph*)]
			\item \(\chi_\lambda=\chi_\mu\),
			\item \(U(\lambda)=U(\mu)\),
			\item \(\lambda=\mu\),
			\item \(V(\lambda)\cong V(\mu)\),
                \item \(\ch V(\lambda)=\ch V(\mu)\).
		\end{enumerate*}
	\end{lemma}
	\begin{proof}
		The equivalence of $(a)$ and $(b)$ directly follows from Equation \eqref{eqnnc}. Assume (b). Put \(\lambda+\rho=\eta\). For a simple reflection \(w=s_\alpha\) corresponding to the even simple root \(\alpha\) we have
		\begin{alignat*}{2}
			U(\lambda) &=-e^{s_\alpha\eta-\eta} &+\sum_{s_\alpha\not=w\in\mathscr{W}}(-1)^{\ell(w)}e^{w\eta-\eta}\phantom{.}\\
			&=-X_\alpha^{\langle\eta,\alpha\rangle} &+\sum_{s_\alpha\not=w\in\mathscr{W}}(-1)^{\ell(w)}e^{w\eta-\eta}.
		\end{alignat*}
		Since no monomial of the form \(X_\alpha^m\) appears in the remaining summation, it follows that for any even simple root \(\beta\) we have,
		\begin{displaymath}	 
                X_\alpha^{\langle\lambda+\rho,\beta\rangle}=X_\alpha^{\langle\mu+\rho,\beta\rangle}.
		\end{displaymath}
  This implies that $\langle\lambda+\rho,\beta\rangle=\langle\mu+\rho,\beta\rangle$ and hence we have 
		 $(\lambda,\beta)=(\mu,\beta)$ for all even simple root $\beta$.
		As the restriction of the invariant form of \g on \h is non degenerate (\(\mathfrak{g}\) is basic), this implies that \(\lambda=\mu\) in the  subspace spanned by the even simple roots. Since the action of \(\lambda,\mu\) is defined to be zero on the center of \(\g_0\), which is one-dimensional for types $A(m,n)$ and $C_n$, it follows that \(\lambda=\mu\) in \(\h^*\). 
  
 To prove (c) implies (b), we observe that the monomial $X(w,\lambda)$ appearing in the expression of  $U(\lambda)$ in \eqref{numer} is a product of \(X_\alpha^{c_{\alpha}(w)}\), \(\alpha\in\Pi_0\). By the proof of \autoref{sopr}, the exponent $c_{\alpha}(w)$ of \(X_\alpha\) depends only on the integers $\langle\lambda+\rho,\gamma\rangle$ where \(\gamma\in\Pi_0\) is arbitrary. Therefore, equality of \(\lambda\) and \(\mu\) would imply equality of all such exponents. In other words, we have $X(w,\lambda)=X(w,\mu)$ for every \(w\in\W\); and consequently, $U(\lambda)=U(\mu)$.
  The equivalence of (c) and (d), and (d) and (e) follow from ~\cite{MR0519631}*{Proposition 2.2} and ~\cite{MR2776360}*{Proposition 4.2} respectively.
	\end{proof}

	\section{Factorization of \texorpdfstring{$U(\lambda)$}{lg}}\label{secfacto}
	\subsection{} Consider the element \(\lambda+\rho-w(\lambda+\rho)\) for any \(w\in\W\). By \autoref{sopr}, this element can be written as:
	\begin{displaymath}
		\lambda+\rho-w(\lambda+\rho)=\sum_{\alpha\in\Pi_0}c_\alpha(w)\alpha.\quad c_\alpha(w)\in\mathbb{Z}_{\geqslant0}
	\end{displaymath}
	By definition and Equation \eqref{numer}, we have 
	\begin{gather}
		U(\lambda)=\sum_{w\in\mathscr{W}}(-1)^{\ell(w)}X(w,\lambda)
		=1-\left(-\sum_{1\not=w\in\mathscr{W}}(-1)^{\ell(w)}X(w,\lambda)\right)\label{oneminuszeta}\
		=1-\zeta,
	\end{gather}
	where \(\zeta\) stands for the term in parenthesis. Therefore, \(-\log U(\lambda)=\sum_{k\geqslant1}\zeta^k/k\) and no monomial in this expansion can include an odd root in its support. In other words, we have to work only with the set of even simple roots \(\Pi_0\). Note that the subgraph spanned by \(\Pi_0\) is not always connected; in fact, it is connected only for types \(A(n,0)\), \(B(0,n)\), \(C_n\), \(F(4)\), \(G(3)\); and for other types, it is union of two connected components. 
 
 \subsection{} We now show that \(U(\lambda)\) factors in accordance with the number of connected components of \(\Pi_0\). Recall that the Weyl group \W of \g is defined to be the Weyl group of the even part \(\g_0\). If \(\Pi_0\) is union of two connected components, say \(\Pi_0=C_1\cup C_2\), then \W will be a direct product \(\W=\W_1\times\W_2\) of two subgroups where \(\W_i\) is the group generated by simple reflections \(\{s_{\alpha_i}\colon\alpha\in C_i\}\) for \(i=1,2\).
	\begin{proposition}{\label{propulamfac}}
		Put $\eta=\lambda+\rho$. Then with the above notation we have
		\begin{gather}
			U(\lambda)=\left(\sum_{u\in \W_1}(-1)^{\ell(u)}e^{u\eta-\eta}\right)\cdot\left(\sum_{v\in \W_2}(-1)^{\ell(v)}e^{v\eta-\eta}\right)
		\end{gather}
	\end{proposition}
	\begin{proof}
		Let \(C_1,C_2\) be the two connected components of \(\Pi_0\) and let \(\W_i\) be the group generated by the simple reflections in \(C_i\) for \(i=1,2\). Since the roots belonging to \(C_1\) are mutually orthogonal to those in \(C_2\), every \(w\in\W\) can be written uniquely as \(w=uv\), for some \(u\in\W_1,v\in\W_2\). Therefore, to prove the proposition, it suffices to show that \(e^{w\eta-\eta}=e^{u\eta-\eta}\cdot e^{v\eta-\eta}\). Let \(u\eta-\eta=\sum_{\alpha\in C_1}c_\alpha(u)\alpha\) and \(v\eta-\eta=\sum_{\beta\in C_2}d_\beta(v)\beta\). We then have,
		\begin{align*}
			w\eta-\eta &=uv\eta-\eta=u(v\eta-\eta)+u\eta-\eta\\
			&=u\left(\sum_{\beta\in C_2}d_\beta(v)\beta\right)+\sum_{\alpha\in C_1}c_\alpha(u)\alpha
			=\sum_{\beta\in C_2}d_\beta(v)\beta+\sum_{\alpha\in C_1}c_\alpha(u)\alpha\\
			&=(v\eta-\eta)+(u\eta-\eta).
		\end{align*}
		The penultimate equality holds because \(u\) can be written as $u=s_{\alpha_1}s_{\alpha_2}\cdots s_{\alpha_k}$ where each \(\alpha_i\in C_1\); and as \(C_1\) is orthogonal to \(C_2\), each \(s_{\alpha_i}\) fixes \(\beta\) for $\beta \in C_2$.
	\end{proof}
        We now define a monomial that will be of particular importance. Let \(C\subseteq\Pi_0\) be any subset and let \(\lambda\) be a typical dominant integral weight. We define
        \begin{displaymath}\label{reg monomial}
            X^\lambda(C)\coloneq\prod_{\alpha\in C}X_\alpha^{\langle\lambda+\rho,\alpha\rangle}.
        \end{displaymath}

	\begin{proposition}{\label{mainprop}}
		Let \(\lambda\) be a typical dominant integral weight of \g, and \(C_1,C_2\) be the two connected components of \(\Pi_0\). For \(i\in\{1,2\}\), denote by \(U_i(\lambda)\), the factor corresponding to \(C_i\) of \(U(\lambda)\). Then we have:
		\begin{enumerate}[label=(\alph*)]
			\item The support of any monomial, say \(\prod_{\alpha\in\Pi_0}X_\alpha^{c_\alpha}\) that appears in \ulog with nonzero coefficient is contained in \(C_i\), i.e: \[\supp\left(\prod_{\alpha\in\Pi_0}X_\alpha^{c_\alpha}\right)\coloneqq\{\alpha\in\Pi_0\colon c_\alpha\not=0\}\subseteq C_i.\]\label{support}
			\item For any \(C\subseteq C_i\), the coefficient of the monomial \(X^\lambda(C)\) depends only on \(C\).\label{dpndonlyonc}
			\item \(X^\lambda(C)\) appears in \ulog with nonzero coefficient if and only if \(C\) is a connected subset of \(C_i\). In particular, \(X^\lambda(C_i)\) appears in \ulog.
		\end{enumerate}
	\end{proposition}
	\begin{proof}
		As in Equation \eqref{oneminuszeta}, we write \(U_i(\lambda)=1-\tau\), where \(\tau=-\sum_{1\not=w\in\W_i}(-1)^{\lw} X(w,\lambda)\). Since every \(w\in\W_i\) is a product of simple reflections \(s_\alpha\), where \(\alpha\in C_i\), the proof of \autoref{sopr} shows that the support of \(X(w,\lambda)\) for any \(w\in\W_i\) is contained in \(C_i\). 
  
  Now \(\ulog=\sum_{k\geqslant 1}\sfrac{\tau^k}{k}\), and every power of \(\tau\) is just a sum of products of \(X(w,\lambda)\)'s such that \(w\) varying over $\W_i$, it follows that the support of any monomial in \ulog is a subset of \(C_i\). This proves (a).

  To show (b), we write \(\tau=\tau_1+\tau_2\) with
		\begin{gather*}
			\tau_1\coloneqq -\sum_{w\in\mathcal{I}}(-1)^{\lw}X(w,\lambda)\quad\text{and}\quad
		\tau_2\coloneqq -\sum_{w\notin\mathcal{I}}(-1)^{\lw}X(w,\lambda).
                \intertext{Then we have,}
                \ulog=\sum_{k\geqslant 1}\frac{(\tau_1+\tau_2)^k}{k}.
		\end{gather*}
  Recall that for a subset $C$ of $\Pi_0$, \(X^\lambda(C)=\prod_{\alpha\in C}X_\alpha^{\langle\lambda+\rho,\alpha\rangle}\). \autoref{keylem}, Part \ref{c bet gret a bet}  shows that \(\tau_2\) does not contribute to the appearance of \(X^\lambda(C)\) in the expansion of \ulog. In other words, the coefficient of \(X^\lambda(C)\) in \(\sum_{k\geqslant1}\sfrac{\tau^k}{k}\) is same as that in \(\sum_{k\geqslant1}\sfrac{\tau_1^k}{k}\). Therefore, it is sufficient to compute the coefficient of \(X^\lambda(C)\) in \(\sfrac{\tau_1^k}{k}\). We have that
		\begin{displaymath}
			\tau_1^k=(-1)^k\sum_{w_j\in\mathcal{I}}(-1)^{\sum\ell(w_j)}\prod_{j=1}^{k}X(w_j,\lambda).
		\end{displaymath}
		By \autoref{keylem}, Part \ref{c alph eq a alph}, the product $\prod_{j=1}^{k}X(w_j,\lambda)$ equals \(X^\lambda(C)\) only when \(\cup_{j=1}^kI(w_j)=C\); and each \(I(w_j)\) is totally disconnected with \(I(w_j)\cap I(w_l)=\emptyset\) for every \(i\not=l\). By \autoref{def k partition}, this means that $(I(w_1),\dots,I(w_k))$ is a \(k\)-partition of \(C\). In particular, for this \(k\)-partition the coefficient of \(X^\lambda(C)\) in \(\tau_1^k\) is $(-1)^k(-1)^{\sum\ell(w_j)}$. As each \(I(w_j)\) is totally disconnected, \(\sum_{j=1}^k\ell(w_j)=\ell(w_1\cdots w_k)\). If we denote the \(k\)-partition $(I(w_1),\dots,I(w_k))$ by \(\mathcal{J}\), then \(w_1\cdots w_k\) is just \(w(\mathcal{J})\); and \((-1)^{\ell(w(\mathcal{J}))}=(-1)^{|C|}\). We obtain that, the coefficient of \(X^\lambda(C)\) in \ulog is:
		\begin{displaymath}
			\sum_{k\geqslant1}\sum_{\mathcal{J}\in P_k(\mathcal{C})}\frac{(-1)^k\cdot(-1)^{\ell{(w(\mathcal{J}))}}}{k}=(-1)^{|C|}\sum_{k=1}^{|C|}(-1)^k\frac{c_k(\mathcal{C})}{k}=k(C),
		\end{displaymath}
		where \(\mathcal{C}\) is the graph spanned by \(C\). By ~\cite{MR2980495}*{Proposition 2}, \(k(C)\) is a positive integer if and only if \(\mathcal{C}\) is connected, otherwise it is zero. Evidently \(k(C)\) does not depend on \(\lambda\), and it is determined completely by \(C\). This completes the proof of (b) and (c).
	\end{proof}
	\begin{remark}
		The coefficient \(k(C)\) above is actually 1 when $C$ is connected ~\cite{MR2980495}*{Corollary 1}.
	\end{remark}
 
 Following ~\cite{MR4343717}, we define for a connected subset \(C\) of \(\Pi_0\), a linear operator \(\Theta_C\colon\mathcal{A}\to\mathcal{A}\) by
 \begin{displaymath}\label{theta c}
     f=\sum_{\boldsymbol{m}}\boldsymbol{X^m}\rightsquigarrow\Theta_C(f)\coloneqq\sum_{\substack{\boldsymbol{m}\\ \supp(\boldsymbol{m})=C}}\boldsymbol{X^m},
 \end{displaymath}
 where \(\boldsymbol{m}=(m_\alpha\colon\alpha\in\Pi_0)\) is an $n$-tuple, where $n=\text{Card}(\Pi_0)$, \(\supp(\boldsymbol{m})\coloneqq\{\alpha\in\Pi_0\colon m_\alpha\ne 0\}\)  and $\boldsymbol{X^m}=X_{\alpha_1}^{m_{\alpha_1}}\cdots X_{\alpha_t}^{m_{\alpha_t}}$ for $\boldsymbol{m}=(m_{\alpha_1},\dots, m_{\alpha_t})$.

 Proof of the following proposition follows directly from \autoref{mainprop}.
	\begin{proposition}{\label{lowestdeg}}
		Let \(\lambda\) be a typical dominant integral weight and let \(C\) be a connected component of \(\Pi_0\). Then
		\begin{gather}
			\Theta_C(-\log U_i(\lambda))=k(C)X^\lambda(C)+\text{monomials of degree $>$ }\deg X^\lambda(C).
		\end{gather}
		where the constant \(k(C)\) depends only on \(C\) and \(\deg X^\lambda(C)\coloneqq\sum_{\alpha\in C}\langle\lambda+\rho,\alpha\rangle\).
	\end{proposition}
	The following lemma gives conditions for which \(U_i(\lambda)\) and \(U_j(\mu)\) are equal.
 
	\begin{lemma}
		Let \(\lambda,\mu\) be typical dominant integral weights and let \(C_1,C_2\) be the two connected components of \(\Pi_0\). Then following statements are equivalent:
		\begin{enumerate}
			\item \(X^\lambda(C_i)=X^\mu(C_j)\),
			\item \(C_i=C_j\) and \(\langle\lambda,\alpha_k\rangle=\langle\mu,\alpha_k\rangle\) for every \(\alpha_k\in C_i\),
			\item \(U_i(\lambda)=U_j(\mu)\).
		\end{enumerate}
	\end{lemma}
    \begin{proof}
			Assume (1). Then all the variables with their exponents must be same; and as a consequence, their corresponding supports are equal, hence (2) follows from (1). 
   
   Now if $C_i=C_j$ we get that \(\W_i=\W_j\), where $\W_i$ (resp. $\W_j$) is the group generated by $s_{\alpha}$, $\alpha \in C_i$ (resp. $\alpha \in C_j$). By \autoref{sopr} the exponents of all the monomials appearing in \(U_i(\lambda)\) (resp. $U_j(\mu)$) depend only on the integers \(\langle\lambda+\rho,\alpha_k\rangle\) (resp. $\langle\mu+\rho,\alpha_k\rangle$) for all \(\alpha_k\in C_i=C_j\), whence (3) follows from (2). Finally, assume (3). After applying $-\log$ and then \(\Theta_{C_i}\) to both sides of the equation, we get that 
   \begin{gather}
       \Theta_{C_i}(-\log U_i(\lambda))=\Theta_{C_i}(-\log U_j(\mu)).\label{xlamxmu}
   \end{gather}
   By \autoref{lowestdeg} we know that \(X^\lambda(C_i)\) must occur in \(\Theta_{C_i}(-\log U_j(\mu))\) with nonzero coefficient. On the other hand, support of any monomial appearing in $-\log U_j(\mu)$ is a subset of \(C_j\) by \autoref{mainprop}, Part \ref{support}. This is possible only when \(C_i=C_j\).  Then using Proposition \ref{lowestdeg} once again we get that \(X^\lambda(C_i)=X^\mu(C_j)\).
		\end{proof}
	
	\section{The main theorem}\label{secmainthm}
	We have seen that \(U(\lambda)\) factorizes as a product: \(U(\lambda)=U_1(\lambda)\cdot U_2(\lambda)\) of two factors. Let \(\lambda_1,\dots,\lambda_r,\mu_1,\dots,\mu_s\) be typical dominant integral weights and suppose that the following equality holds:
	\begin{displaymath}
		U(\lambda_1)\cdots U(\lambda_r)=U(\mu_1)\cdots U(\mu_s).
	\end{displaymath}
	After factoring them further one obtains,
	\begin{gather}{\label{factorization}}
		U_1(\lambda_1)U_2(\lambda_1)\cdots U_1(\lambda_r)U_2(\lambda_r)=U_1(\mu_1)U_2(\mu_1)\cdots U_1(\mu_s)U_2(\mu_s).
	\end{gather}
 
	\begin{theorem}{\label{permfac}}
		With the notations as above, we have \(r=s\) and the factors in \eqref{factorization} are all equal up to a permutation, i.e., there exists a bijection \(\sigma\) of \(\{1,2\}\times\{1,\dots,r\}\) such that \(U_i(\lambda_p)=U_j(\lambda_q)\), where \(i,j\in\{1,2\}\), \(p,q\in\{1,\dots,r\}\) and \(\sigma(i,p)=(j,q)\).
	\end{theorem}
	\begin{proof}
		Applying \(-\log\) on both sides of Equation \eqref{factorization}, we obtain
		\begin{gather}
			\sum_{p=1}^{r}\sum_{i=1}^{2}-\log U_i(\lambda_p)=\sum_{q=1}^{s}\sum_{j=1}^{2}-\log U_j(\mu_q).\label{logsum}
			\intertext{Pick a connected component of \(\Pi_0\), say \(C_1\). We choose \(\lambda_k\) such that the monomial \(X^{\lambda_k}(C_1)\) is of minimal degree among all the monomials with support \(C_1\) in the left hand side of Equation \eqref{logsum}. Applying the operator \(\Theta_{C_1}\) to Equation \eqref{logsum} yields:}
			\Theta_{C_1}\left(\sum_{p=1}^{r}\sum_{i=1}^{2}-\log U_i(\lambda_p)\right)=\Theta_{C_1}\left(\sum_{q=1}^{s}\sum_{j=1}^{2}-\log U_j(\mu_q)\right).\label{avlogsum}
		\end{gather}
		By \autoref{mainprop}, the support of any monomial that appears in \(-\log U_i(\lambda_p)\) is contained in \(C_i\). Therefore, if \(i\neq 1\), then \(\Theta_{C_1}(-\log U_i(\lambda_p))=0\). On the other hand if \(i=1\), then \(X^{\lambda_p}(C_i)\) is the minimal degree monomial in \(\Theta_{C_1}(-\log U_i(\lambda_p))\) which is nonzero in this case. 
  
  Since the monomial \(X^{\lambda_k}(C_1)\) occurs in the left hand side of Equation \eqref{avlogsum}, by minimality of degree, it must appear in the right hand side with nonzero coefficient. In other words, there exists \(1\leqslant q\leqslant s\) and \(j\in\{1,2\}\) such that \(X^{\lambda_k}(C_1)=X^{\mu_q}(C_j)\). \autoref{mainlem} now gives credence to the conclusion that \(U_1(\lambda_k)=U_j(\mu_q)\). To obtain the desired claim we cancel \(U_1(\lambda_k)\ \text{and}\ U_j(\mu_q)\) in Equation \eqref{factorization} and proceed by induction. 
  
  To prove \(r=s\), we observe that the number of factors in the left hand side of Equation \eqref{factorization} is \(2r\) and that in the right hand side is \(2s\). So if \(r>s\), then on the left hand side we have a product of \(2(r-s)\) number of factors whereas in the right hand side we have $1$, which is a contradiction. So we conclude that \(r=s\).
	\end{proof}
 
	We now give a sufficient condition for which two tensor products of irreducible typical representations are isomorphic to each other.
	\begin{theorem}{\label{Thm:tnsrpdt}}
		 Let $\lambda_1,\dots,\lambda_r,\mu_1,\dots,\mu_s$ be typical dominant integral weights and assume that
		\begin{gather}{\label{tnsrpdt}}
			V(\lambda_1)\otimes\dots\otimes V(\lambda_r)\cong V(\mu_1)\otimes\dots\otimes V(\mu_s).
		\end{gather}
		Then \(r=s\). Suppose the bijection \(\sigma\) in \autoref{permfac} has the following additional property:
		if \(U_i(\lambda_p)=U_j(\mu_q)\), then \(\sigma(j,p)=(i,q)\) for all \(i,j\in\{1,2\}\), \(p,q\in\{1,\dots,r\}\) and \(\sigma(i,p)=(j,q)\). Then \(V(\lambda_p)\cong V(\mu_q)\) for all \(p,q\in\{1,\dots,r\}\).
	\end{theorem}
	\begin{proof}
		To prove \(r=s\), we assume on contrary that \(r>s\). The maximal weights that occur on both sides of the isomorphism \eqref{tnsrpdt} are equal, i.e, \(\sum_{i=1}^{r}\lambda_i=\sum_{j=1}^{s}\mu_j\coloneqq\gamma,\) say. Taking formal character on both sides of Equation \eqref{tnsrpdt} yields:
		\begin{gather}
		   \ch V(\lambda_1)\cdots\ch V(\lambda_r)=\ch V(\mu_1)\cdots\ch V(\mu_s).\notag
            \intertext{Multiplying both sides of the above equation by $e^{-\gamma}$ and grouping the corresponding highest weights we obtain:}
            \prod_{i=1}^{r}e^{-\lambda_i}\cdot\ch V(\lambda_i)=\prod_{j=1}^{s}e^{-\mu_j}\cdot\ch V(\mu_j).\notag
            \intertext{By Equation \eqref{eqnnc}, this simplifies to}
            U(\lambda_1)\cdots U(\lambda_r)D^{r-s}=U(\mu_1)\cdots U(\mu_s),\quad\text{where $D=\tfrac{\prod_{\alpha\in\Phi^+_1}(1+e^{-\alpha})}{\prod_{\alpha\in\Phi^+_0}(1-e^{-\alpha})}$}.\notag
            \intertext{As every $U(\lambda_k)$ (resp. $U(\mu_l)$) is a product of two factors, we get that:}
			U_1(\lambda_1)U_2(\lambda_1)\cdots U_1(\lambda_r)U_2(\lambda_r)D^{r-s}=U_1(\mu_1)U_2(\mu_1)\cdots U_1(\mu_s)U_2(\mu_s).\label{spltfac}
		\end{gather}
		By \autoref{permfac}, we conclude that $U_i(\lambda_p)=U_j(\mu_q)$, for some $j\in\{1,2\}$ and $q\in\{1,\dots,s\}$. Therefore, arguing as in the proof of \autoref{permfac} we get a contradiction. Hence \(r=s\). 
		
		Now suppose the bijection \(\sigma\) satisfies the condition mentioned in the hypotheses of the theorem. Then we get that \[\prod_{i=1}^{2}U_i(\lambda_p)=\prod_{j=1}^{2}U_j(\mu_q).\]
		This implies that \(U(\lambda_p)=U(\mu_q)\).
		By \autoref{mainlem}, we get that \(V(\lambda_p)\cong V(\mu_q)\). 
	\end{proof}

 The following corollary says that the unique factorization of tensor products of irreducible finite dimensional typical representations holds for superalgebras of types \(B(0,n),C_n,F(4)\) and \(G(3)\).
 
	\begin{corollary}\label{unique-coro}
		Let \g be of type \(A(n,0),B(0,n),C_n,F(4)\) or \(G(3)\) and let \(\lambda_1,\dots,\lambda_r,\mu_1,\dots,\mu_s\) be typical dominant integral weights of \g . Suppose we are given an isomorphism of representations: 
		\begin{gather}{\label{eqtnsrprdt}}
			V(\lambda_1)\otimes\dots\otimes V(\lambda_r)\cong V(\mu_1)\otimes\dots\otimes V(\mu_s).
		\end{gather}
		 Then $r=s$, and there is a permutation \(\sigma\) of \(\{1,\dots,r\}\) such that \(V(\lambda_k)\cong V(\mu_{\sigma(k)})\) for all $k$.
	\end{corollary}
	\begin{proof}
		By \autoref{Thm:tnsrpdt}, we have $r=s$. Proceeding as in its proof we further obtain that
		\begin{gather}
			U(\lambda_1)\cdots U(\lambda_r)=U(\mu_1)\cdots U(\mu_r)
		\end{gather}
		The standard Dynkin diagrams of these types are given in ~\cite{MR0519631}*{Table 1, page 606}. Evidently \(\Pi_0\) is connected in all of these cases as the node corresponding to the odd simple root appears at the edge of the diagrams. So removal of this node will not render the diagram disconnected. This means that \(U_2(\lambda_k)=1\) and \(U_2(\mu_l)=1\) for all $k,l\in\{1,\dots,r\}$. Then the corollary follows from \autoref{permfac} and \autoref{mainlem}.
	    
		For type \(B(0,n)\) an alternate proof can be given. It is known that if \(\lambda\) is the highest weight of an irreducible representation of \(\mathfrak{osp}(1,2n)\), then it is also the highest weight of a non spinorial irreducible representation of \(\mathfrak{so}(2n+1)\); and characters of both the representations are same. For details we refer to ~\cite{MR0648354}. Taking formal character on both sides of the isomorphism \eqref{eqtnsrprdt} yields:
		\begin{displaymath}
			\ch V(\lambda_1)\cdots\ch V(\lambda_r)=\ch V(\mu_1)\cdots\ch V(\mu_r).
		\end{displaymath}
		The above discussion shows that this equality can be considered an equality for \(\mathfrak{so}(2n+1)\)-modules. By ~\cite{MR2123935}*{Theorem 1}, this implies that \(V(\lambda_i)\cong V(\mu_{\sigma(i)})\) as \(\mathfrak{so}(2n+1)\)-modules for some permutation \(\sigma\) of \(\{1,\dots,r\}\); so $\ch V(\lambda_i)=\ch V(\mu_{\sigma(i)})$. Hence \(V(\lambda_i)\cong V(\mu_{\sigma(i)})\) as \(\mathfrak{osp}(1,2n)\)-modules.
	\end{proof}

    The example below illustrates that the additional hypothesis on the bijection $\sigma$  in \autoref{Thm:tnsrpdt} is essential for the conclusion of unique factorization of the tensor products.
 
	\begin{example}\label{example}
		We take \(\g=\mathfrak{sl}(3,2)\). In this case \(\Pi_0=\{\alpha_1,\alpha_2,\alpha_3\}\), and \(C_1=\{\alpha_1,\alpha_2\}\), \(C_2=\{\alpha_3\}\) are the two connected components of \(\Pi_0\). The subgroups of the Weyl group of \g generated by \(C_1\) and \(C_2\) are \(\W_1=\{1,s_1,s_2,\ s_1s_2,\ s_2s_1,\ s_1s_2s_1\}\) and \(\W_2=\{1,s_3\}\) respectively. Here \(s_i\) is the simple reflection corresponding to \(\alpha_i\). The set of all positive odd roots \(\Phi_1^+\) is given by:
		\begin{displaymath}
			\Phi_1^+=\{\varepsilon_i-\delta_j\colon 1\leqslant i\leqslant 3;\ 1\leqslant j\leqslant2\},
		\end{displaymath}
		where \(\varepsilon_i(\text{diag}(a_1,a_2,a_3))=a_i\) and \(\delta_j(\text{diag}(b_1,b_2))=b_j\) for all $i,j$.
  
 The sum of all elements in \(\Phi_1^+\) is \(\tau\coloneq 2(\varepsilon_1+\varepsilon_2+\varepsilon_3)-3(\delta_1+\delta_2) \in\h^*\). Here all odd roots are isotropic. We have that \((\tau,\varepsilon_i-\delta_j)=5\). Consider now the following weights:
		\begin{displaymath}
			\begin{gathered}
				\lambda_1=\omega_1+2\omega_2+3\omega_3+\tau\\
				\lambda_2=\omega_1+4\omega_2+5\omega_3+\tau
			\end{gathered}
			\qquad\text{and}\qquad
			\begin{gathered}
				\mu_1=\omega_1+4\omega_2+3\omega_3+\tau\\
				\mu_2=\omega_1+2\omega_2+5\omega_3+\tau
			\end{gathered}
		\end{displaymath}
		where \(\omega_i\) is the fundamental dominant weight corresponding to \(\alpha_i\). All of these weights are typical dominant integral (cf. \autoref{rmkdominant}). After computing and factoring the normalized Weyl numerators we find:
		\begin{displaymath}
			\begin{gathered}
				U_1(\lambda_1)=1-X_{\alpha_1}^2-X_{\alpha_2}^3+X_{\alpha_1}^5X_{\alpha_2}^3 +X_{\alpha_1}^2X_{\alpha_2}^5-X_{\alpha_1}^5X_{\alpha_2}^5\\
				U_1(\lambda_2)=1-X_{\alpha_1}^2-X_{\alpha_2}^5+X_{\alpha_1}^7X_{\alpha_2}^5+X_{\alpha_1}^2X_{\alpha_2}^7-X_{\alpha_1}^7X_{\alpha_2}^7\\
                U_1(\mu_1)=1-X_{\alpha_1}^2-X_{\alpha_2}^5+X_{\alpha_1}^7X_{\alpha_2}^5+X_{\alpha_1}^2X_{\alpha_2}^7-X_{\alpha_1}^7X_{\alpha_2}^7\\
				U_1(\mu_2)=1-X_{\alpha_1}^2-X_{\alpha_2}^3+X_{\alpha_1}^5X_{\alpha_2}^3+X_{\alpha_1}^2X_{\alpha_2}^5-X_{\alpha_1}^5X_{\alpha_2}^5
			\end{gathered}
                \qquad\text{and}\qquad
                \begin{gathered}
				U_2(\lambda_1)=1-X_{\alpha_3}^4\\				
				U_2(\lambda_2)=1-X_{\alpha_3}^6\\
				U_2(\mu_1)=1-X_{\alpha_3}^4\\				
				U_2(\mu_2)=1-X_{\alpha_3}^6
			\end{gathered}
		\end{displaymath}
		Evidently,
		\begin{align*}
			U_1(\lambda_1)U_2(\lambda_1)U_1(\lambda_2)U_2(\lambda_2)&=U_1(\mu_1)U_2(\mu_1)U_1(\mu_2)U_2(\mu_2)\\
			\Rightarrow U(\lambda_1)U(\lambda_2)&=U(\mu_1)U(\mu_2).
			\intertext{Using the fact that $\lambda_1+\lambda_2=\mu_2+\mu_2$, and multiplying both sides by $D_1/D_2$ (cf. Equation \eqref{d1d2}) we obtain}
			\ch V(\lambda_1)\ch V(\lambda_2)&=\ch V(\mu_1)\ch V(\mu_2).
			\intertext{From here we get that}
			V(\lambda_1)\otimes V(\lambda_2)&\cong V(\mu_1)\otimes V(\mu_2).
		\end{align*}
		So, unique decomposition of tensor products does not hold.
        \end{example}

        \section{Atypical representations}\label{sec atyp}
         \subsection{} In this section we focus on some atypical representations of $\mathfrak{sl}(m+1,n+1),\ C_{n+1}=\mathfrak{osp}(2,2n),\ G(3)$ and $F(4)$ for which a character formula is known in the literature. A weight $\lambda$ is called \emph{singly} atypical if there is a unique \(\gamma\in\Phi_1^+\) such that \((\lambda+\rho,\gamma)=0\). In this case, $\lambda$ is said to have atypicality type $\gamma$. For $\g=\mathfrak{sl}(m+1,n+1)$, the class of singly atypical finite dimensional irreducible \g-modules admit a character formula closely resembling that of the typical ones. Details can be found in ~\cite{MR1063989}. It is known that any dominant integral weight of type $C_{n+1}$, and the exceptional Lie superalgebras $F(4)$ and $G(3)$ is either typical or singly atypical; and a character formula for singly atypical finite dimensional irreducible representations for type $C_n$ can be found in ~\cite{MR1092559}. The same for the exceptional Lie superalgebras is obtained in ~\cite{MR3253284}.

         Our goal of this section is to establish unique decomposition of tensor products of singly atypical finite dimensional irreducible modules for the above mentioned Lie superalgebras. Let \g be any such Lie superalgebra and $\lambda\in\h^*$ be a dominant integral singly atypical weight of type $\beta$. For $G(3)$ and $F(4)$, we first assume that $\lambda\neq \lambda_1,\lambda_2$ where $\lambda_1$ and $\lambda_2$ are the special weights (\textit{see} ~\cite{MR3253284}*{Theorem 2.6} for a description of these two weights). Then from ~\cites{MR1063989,MR1092559,MR3253284}, the character formula for $V(\lambda)$ is given by:
          \begin{align}
            \ch V(\lambda) &=\frac{D_1}{D_2}\sum_{w\in\W} (-1)^{\lw} \frac{1}{1+e^{-w\beta}}\cdot e^{w(\lambda+\rho)}.\label{char}
            \intertext{When $\lambda$ is either $\lambda_1$ or $\lambda_2$ for type $G(3)$ and $F(4)$, the character formula for $V(\lambda)$ is given by:}
            \ch V(\lambda) &=\frac{D_1}{D_2} \sum_{w\in\W}\frac{(-1)^{\lw}}{2}\frac{2+e^{-w\beta}}{1+e^{-w\beta}}\cdot e^{w(\lambda+\rho)},\label{char lam1 lam2}
            \end{align}
            where \(D_1=\prod_{\alpha\in\Phi^+_1}(e^{\alpha/2}+e^{-\alpha/2})\) and \(D_2=\prod_{\alpha\in\Phi^+_0}(e^{\alpha/2}-e^{-\alpha/2})\).
            
            As in \autoref{secprelims}, we define the normalized Weyl numerator $U(\lambda)$ for $\lambda\neq \lambda_1,\lambda_2$ by:
            \begin{gather}
                U(\lambda)\coloneqq \sum_{w\in\W} \frac{(-1)^{\lw} e^{w(\lambda+\rho)-(\lambda+\rho)}}{1+e^{-w\beta}}=\sum_{w\in\W}(-1)^{\lw}\frac{1}{1+e^{-w\beta}}\cdot X(\lambda,w).
                \intertext{For $\lambda=\lambda_1,\lambda_2$ in type $G(3)$ and $F(4)$, we define:}
                U(\lambda)\coloneqq \sum_{w\in\W}\frac{(-1)^{\lw}}{2}\frac{2+e^{-w\beta}}{1+e^{-w\beta}}\cdot X(\lambda,w)
            \end{gather}
            where $X_\alpha=e^{-\alpha}$ for $\alpha\in\Pi_0$, and $X(\lambda,w)=\prod_{\alpha\in\Pi_0}X_{\alpha}^{c_\alpha(w)}=e^{w(\lambda+\rho)-(\lambda+\rho)}$. 
            
            The definition of the normalized character is same as before. Equation \eqref{eqnnc} remains valid here as well. Let us recall that the monomial $X^\lambda(\Pi_0)$ is defined by $X^\lambda(\Pi_0)\coloneq\prod_{\alpha\in \Pi_0}X_\alpha^{\langle\lambda+\rho,\alpha\rangle}$ (\textit{see} \vpageref{reg monomial}). As before, our primary task is to show that the coefficient of this monomial in the expansion of $-\log U(\lambda)$ is nonzero. We achieve this by means of case by case consideration. First we record the following lemma that shows $X^\lambda(\Pi_0)$ does determine the representation $V(\lambda)$.

            \begin{notation}
            In what follows we denote the monomial $X^\lambda(\Pi_0)$ by just $X^\lambda$ for every $\lambda\in\h^*$.
        \end{notation}

        \begin{lemma}{\label{ x lam to u lam atyp}}
            Let \g be any Lie superalgebra among $\mathfrak{sl}(m+1,n+1),\ \mathfrak{osp}(2,2n),\ G(3)$, or $F(4)$. Let $\lambda,\mu\in\h^*$ be dominant integral singly atypical weights of the same type, say $\gamma$. Then the following statements are equivalent:
           
            \begin{enumerate*}[label=(\alph*)]
                \item $X^{\lambda}=X^{\mu}$,
                \item $\lambda=\mu$,
                \item $U(\lambda)=U(\mu)$.
            \end{enumerate*}
        \end{lemma}
        \begin{proof}
            Assume (a). By definition, this means that $\langle\lambda+\rho,\alpha\rangle=\langle\mu+\rho,\alpha\rangle$, for every $\alpha\in\Pi_0$. Now an argument as in the proof of \autoref{mainlem} gives us $\lambda=\mu$. If we have $\lambda=\mu$ to begin with, then $X^\lambda=X^\mu$ follows form the definition. This proves the equivalence of (a) and (b).

            Suppose (b) is given. Then (c) follows from definition. Now we assume $U(\lambda)=U(\mu)$ where $\lambda,\mu\neq\lambda_1,\lambda_2$ in case when $\g=G(3),F(4)$. For any $\alpha\in\Pi_0$, $X_\alpha^{\langle\lambda+\rho,\alpha\rangle}$ (resp. $X_\alpha^{\langle\mu+\rho,\alpha\rangle}$) is the only monomial of the form $X_\alpha^m$ in $U(\lambda)$ (resp. $U(\mu)$) with coefficient $(1+e^{-\gamma})^{-1}$. This gives that $(1+e^{-\gamma})^{-1}X_\alpha^{\langle\lambda+\rho,\alpha\rangle}=(1+e^\gamma)^{-1}X_\alpha^{\langle\mu+\rho,\alpha\rangle}$ for all even simple roots $\alpha$. So we have $\langle\lambda+\rho,\alpha\rangle=\langle\mu+\rho,\alpha\rangle$ for all $\alpha\in\Pi_0$. Now arguing as in the proof of \autoref{mainlem}, we conclude that $\lambda=\mu$. This shows the equivalence of (b) and (c).

            If $\lambda,\mu$ are among the special weights $\lambda_1,\lambda_2$ for $G(3)$ and $F(4)$, then the above proof works with $(1+e^{-\gamma})^{-1}$ replaced by $\frac{1}{2}(2+e^{-\gamma})(1+e^{-\gamma})^{-1}$.
        \end{proof}

        \begin{remark}
            Notice that omission of the additional condition on the highest weights having same atypicality type in the above lemma poses a technical difficulty. Indeed, suppose $\lambda$ is of type $\gamma_1$ and $\mu$ is of type $\gamma_2$ with $\gamma_1\neq\gamma_2$. Then from $X^\lambda=X^\mu$, we only get equality of numerators of the term corresponding to $w_i\in\W$ in $U(\lambda)$ and $U(\mu)$. Denominator of the term corresponding to $w_i$ in $U(\lambda)$ is $1+e^{-w_i\gamma_1}$ whereas the same in $U(\mu)$ is $1+e^{-w_i\gamma_2}$. Therefore, in this case $X^\lambda=X^\mu$ does not imply $U(\lambda)=U(\mu)$, which precludes us from any conclusion about isomorphism of the corresponding representations.
        \end{remark}

            The proof that the coefficient of $X^\lambda$ in $-\log U(\lambda)$ is nonzero for all \g in \autoref{ x lam to u lam atyp} is given in the following subsections. First we have singled out the $\mathfrak{sl}(m,1)$ case inasmuch as the proofs for the other types closely resemble it. At the end we have treated the superalgebra $\mathfrak{sl}(m+1,n+1)$.
            
        \begin{notation}
            We denote by $\mathcal{F}$ the formal power series algebra $\mathbb{C}[[Z_\gamma\colon\gamma\in\Phi_1]]$, where $Z_\gamma=e^{-\gamma}$.
        \end{notation}
        
        \subsection{} In this subsection we assume that $\g=\mathfrak{sl}(m,1)$ and we show that the coefficient of $X^\lambda$ in $-\log U(\lambda)$ is nonzero. We start by listing down the positive roots of \g:
        \begin{align}
            \Phi_0^+=\{\varepsilon_i-\varepsilon_j\colon 1\leqslant i<j\leqslant m\},\quad &\text{and}\quad \Phi_1^+=\{\varepsilon_i-\delta_1\colon 1\leqslant i\leqslant m\}.\label{roots}
            \intertext{The invariant form on $\h^*$ is determined by}(\varepsilon_i,\varepsilon_j)=\delta_{ij},\quad &\text{and}\quad  (\varepsilon_i,\delta_1)=0.\label{invform}
        \end{align} 
        Denote by $\alpha_i$ (resp. $\beta_i$) the even simple root $\varepsilon_i-\varepsilon_{i+1}$, (resp. the isotropic root $\varepsilon_i-\delta_1$). In this case all odd roots are isotropic. The simple reflection corresponding to $\alpha_i$ is denoted by $s_i$. Let $\lambda\in\h^*$ be a dominant integral weight which is singly atypical of type $\beta=\beta_t$, for some $1\leqslant t\leqslant m$.

        The following Proposition describes the coefficient of $X^\lambda$ in $-\log U(\lambda)$. From the standard Dynkin diagram of $\mathfrak{sl}(m,1)$, we see that $\Pi_0$ is connected. Moreover, in this case both $-\log U(\lambda)$ and $U(\lambda)$ are elements of $\mathcal{F}[[X_\alpha\colon\alpha\in\Pi_0]]$. The proof closely parallels that of \autoref{mainprop}, so we shall be little concise here. We put $K\coloneq k(\Pi_0)=(-1)^{|\mathcal{G}|}\sum_{k=1}^{|\Pi_0|}(-1)^k\frac{c_k(\Pi_0)}{k}$, where $c_k(\Pi_0)$ is the number of $k$-partitions of $\Pi_0$.
        
        \begin{proposition}\label{coef for sl n 1}
            With the notations as above, the coefficient of $X^\lambda$ in $-\log U(\lambda)$ is given by
            \begin{displaymath}
                \begin{dcases}
                    \qquad K\frac{1+Z_{\beta_1}}{1+Z_{\beta_2}} &\quad \text{for $\beta=\beta_1$},\\[3pt]
                    \qquad K\frac{1+Z_{\beta_m}\hfill}{1+Z_{\beta_{m-1}}} &\quad \text{for $\beta=\beta_m$},\\[3pt]
                    \frac{K(1+Z_{\beta_t})^2}{(1+Z_{\beta_{t-1}})(1+Z_{\beta_{t+1}})}  &\quad \text{for $\beta=\beta_t\neq\beta_1,\beta_m$}.
                \end{dcases}
            \end{displaymath}
        \end{proposition}
        \begin{proof}
            In order to calculate $-\log U(\lambda)$, first we express $U(\lambda)$ in the following form:
            \begin{align}
                \begin{aligned}\label{eq 1-xi in atyp}
                    U(\lambda)= \sum_{w\in\W}\frac{(-1)^{\lw} X(\lambda,w)}{1+Z_{w\beta}} &=\frac{1}{1+Z_{\beta}} + \sum_{1\ne w\in\W}\frac{(-1)^{\lw} X(\lambda,w)}{1+Z_{w\beta}}\\
                &=\frac{1}{1+Z_{\beta}}\left(1+\underbrace{\sum_{1\ne w\in\W}(-1)^{\lw}\frac{1+Z_{\beta}\hfill}{1+Z_{w\beta}} X(\lambda,w)}_{-\xi}\right)\\
                &=\frac{1}{1+Z_{\beta}}(1-\xi)
                \end{aligned}
            \end{align}
            We have $\frac{1}{1+Z_{\beta}}=(1+Z_{\beta})^{-1}=1-Z_{\beta}+Z_{\beta}^2-\cdots$, and $\frac{1+Z_{\beta}\hfill}{1+Z_{w\beta}}=(1+Z_\beta)(1-Z_{w\beta}+Z_{w\beta}^2-\cdots)$. As \W leaves $\Phi_1$ invariant, both of these elements are members of $\mathcal{F}$ with constant term 1. Like in the typical case, we write $\xi=\xi_1+\xi_2$ with
            \begin{gather}\label{xi 1 and xi 2}
                \xi_1\coloneqq -\sum_{w\in\mathcal{I}}(-1)^{\lw}\frac{1+Z_{\beta}\hfill}{1+Z_{w\beta}} X(\lambda,w) \quad\text{and}\quad \xi_2\coloneqq -\sum_{w\notin\mathcal{I}}(-1)^{\lw}\frac{1+Z_{\beta}\hfill}{1+Z_{w\beta}} X(\lambda,w),
            \end{gather}
            where $\mathcal{I}=\{1\not=w\in\W\colon I(w)\text{ is totally disconnected}\}$, and $I(w)$ is the subset \{$\alpha\in\Pi_0\colon\alpha$ appears in a reduced expression of $w$\}. Now, $\log U(\lambda)=\sum_{k\geqslant 1}\sfrac{\xi^k}{k}$. Proceeding as in the proof of \autoref{mainprop}, we find that the coefficient of $X^\lambda$ in $-\log U(\lambda
            )$ is same as that in $\sum_{k\geqslant 1}\sfrac{\xi_1^k}{k}$; therefore it is enough to compute the coefficient of $X^\lambda$ in $\xi_1^k$. For any $k\geqslant 1$, we have that
            \begin{gather}{\label{coef of xi k}}
                \xi_1^k= (-1)^k\mathlarger{\sum_{w_i\in\mathcal{I}}}(-1)^{\sum\ell(w_i)} (1+Z_\beta)^k \mathlarger{\prod_{i=1}^k}\frac{X(\lambda,w_i)}{(1+Z_{w_i\beta})}.
            \end{gather}
            By \autoref{keylem}, the product $\prod_{i=1}^k X(\lambda,w_i)$ equals $X^\lambda$ only when $\mathcal{Q}\coloneqq(I(w_1),\dots, I(w_k))$ forms a $k$-partition of $\Pi_0$.

            From the description of the odd roots of \g given in \eqref{roots}, we have that for $t\neq 1,m$
            \begin{alignat*}{3}
                s_i\beta_1=
                \begin{cases}
                    \beta_2 & \text{for $i=1$}\\
                    \beta_1 & \text{for $i\neq 1$}
                \end{cases}
                \quad && \quad
                s_i\beta_m=
                \begin{cases}
                    \beta_{m-1} & \text{for $i=m$}\\
                    \beta_m & \text{for $i\neq m$}
                \end{cases}
                \qquad&&\text{and}\quad
                s_i\beta_t=
                \begin{cases}
                    \beta_{t+1} & \text{for $i=t$}\\
                    \beta_{t-1} & \text{for $i=t-1$}\\
                    \beta_t & \text{for $i\neq t,t-1$}.
                \end{cases}
            \end{alignat*}
            First we discuss the case where $\beta\notin\{\beta_1,\beta_m\}$. As $\cup_{j=1}^k I(w_j)=\Pi_0$, it follows that $\alpha_{t-1}\in I(w_p)$ and $\alpha_t\in I(w_q)$ for some $p,q\in\{1,\dots,k\}$ with $p\neq q$. Also by \autoref{def k partition}, each $I(w_j)$ is totally disconnected, whence $w_j$ is a product of commuting simple reflections. In particular, $w_p=s_{t-1}u$, for some $u\in\W$ such that $\alpha_{t-1}\notin I(u)$. Then $w_p\beta=\beta_{t-1}$. Similarly, $w_q\beta=\beta_{t+1}$, and $w_j\beta=\beta$ for all $j\ne p,q$. This observation immediately implies that for $k\geqslant 2$ we have: (\textit{see} \autoref{coef sl2,1})
            \begin{gather}{\label{y beta sq}}
                \frac{(1+Z_\beta)^k}{\prod_{i=1}^{k}(1+Z_{w_i\beta})}=\frac{(1+Z_\beta)^2}{(1+Z_{\beta_{t-1}})(1+Z_{\beta_{t+1}})}\coloneqq M(\beta),\text{ say.}
            \end{gather}
            Since $\mathcal{Q}$ is a $k$-partition, we get $\sum_{j=1}^k \ell(w_j)=\ell(w_1\cdots w_k)=\ell(w(\mathcal{Q}))$, and $(-1)^{\ell(w(\mathcal{Q}))}=(-1)^{|\Pi_0|}$.
            Note that $M(\beta)$ is independent of $\mathcal{Q}$, therefore summing over all $k$-partitions of $\Pi_0$ we find that the coefficient of $X^\lambda$ in $\sfrac{\xi_1^k}{k}$ is given by:
            \begin{gather*}
                \sum_{\mathcal{Q}\in P_k(\Pi_0)} \frac{(-1)^k\cdot (-1)^{\ell(w(\mathcal{Q}))} M(\beta)}{k}=\frac{(-1)^k \cdot (-1)^{|\Pi_0|} c_k(\Pi_0)M(\beta)}{k}.
                \intertext{Equation \eqref{y beta sq} remains valid for any positive integer $k>1$, so finally we obtain that the required coefficient in $\sum_{k\geqslant 1}\sfrac{\xi_1^k}{k}$ equals:}
                 M(\beta) (-1)^{|\Pi_0|}\sum_{k\geqslant 1}^{|\Pi_0|} \frac{(-1)^k\cdot c_k(\Pi_0)}{k}=M(\beta)k(\Pi_0). 
            \end{gather*}
            This completes the proof for the case when $\beta\notin\{\beta_1,\beta_m\}$.

            Now assume $\beta=\beta_1=\varepsilon_1-\delta_1$. Since $s_i(\beta_1)=\beta_1$ for $i \neq 1$ and $s_1(\beta_1)=\beta_2$, Equation \eqref{y beta sq} now takes the following form:
            \begin{displaymath}
                \frac{(1+Z_\beta)^k}{\prod_{i=1}^{k}(1+Z_{w_i\beta})}=\frac{1+Z_\beta\hfill}{1+Z_{\beta_2}}\coloneqq N(\beta),\text{ say.}\tag{\ref*{y beta sq}a}
            \end{displaymath}
            Arguing as above we conclude that the coefficient of $X^\lambda$ in $-\log U(\lambda)$ is just $N(\beta)k(\Pi_0)$. The case when $\beta=\beta_m=\varepsilon_m-\delta_1$ is similar.
        \end{proof}

        \begin{remark}\label{coef sl2,1}
            In the above proof we have suppressed the fact that Equation \eqref{y beta sq} is not valid for $k=1$. Indeed, the number of 1-partitions is nonzero only when $\g=\mathfrak{sl}(2,1)$. However, in this case we can directly see that the coefficient of $X^\lambda$ in $-\log U(\lambda)$ is $\smash{-\sfrac{(1+Z_\beta\hfill)}{(1+Z_{s\beta})}}$, where $\beta$ is the atypicality type of $\lambda$, and $s$ being the unique simple reflection in \W.
        \end{remark}

        \subsection{} In this subsection we assume that \g is of type $C_{n+1}$, i.e., $\g=\mathfrak{osp}(2,2n)$ for $n>1$. The root system of \g is given by:
        \begin{gather*}
            \Phi_0=\{\pm2\delta_i;\ \pm\delta_i\pm\delta_j\}_{i\neq j},\quad \Phi_1=\{\pm\varepsilon_1\pm\delta_j\}.
            \intertext{The invariant form on $\h^*$ is determined by:}
            (\delta_i,\delta_j)=-\delta_{ij}\quad\text{and}\quad (\varepsilon_1,\delta_j)=0.
            \intertext{The standard simple system is given by:}
            \Pi=\{\delta_1-\delta_2,\dots, \delta_{n-1}-\delta_n,2\delta_n;\  \varepsilon_1-\delta_1\}.
        \end{gather*}
        Denote by $\alpha_j$ the simple root $\delta_j-\delta_{j+1}$, for $1\leqslant j\leqslant n-1$; and we put $\alpha_n=2\delta_n,\ \alpha_{n+1}=\varepsilon_1-\delta_1$. We also have that $\Pi_0=\Pi\backslash\{\alpha_{n+1}\}$, and $\Phi_1^+=\{\varepsilon_1\pm \delta_j\}$ (all are isotropic). The simple reflection corresponding to $\alpha_j$ is denoted by $s_j$. 
        
        Let $\gamma_p\coloneqq\varepsilon_1+\delta_p$ and $\gamma_p'\coloneqq\varepsilon_1-\delta_p$. Consider a dominant integral singly atypical weight $\lambda$ of \g of type $\gamma=\gamma_p$ (the case for $\gamma=\gamma_p'$ is similar). The character formula for $V(\lambda)$ is given by Equation \eqref{char} with $\beta$ replaced by $\gamma$, and the symbols being understood in $\mathfrak{osp}(2,2n)$. Note that $s_j\gamma_p=\gamma_p$ for $j \neq p, p-1$; while $s_{p-1}\gamma_p=\gamma_{p-1},\  s_p\gamma_p=\gamma_{p+1}$ and $s_{n-1}\gamma_n=\gamma_{n-1},\ s_n\gamma_n=\gamma_n'$. Now arguing as in the proof of \autoref{coef for sl n 1}, we obtain the following Proposition.
        
        \begin{proposition}\label{coef for cn}
            With notations as above, the coefficient of $X^\lambda$ in $-\log U(\lambda)$ is given by
            \begin{displaymath}
                \begin{dcases}
                    \qquad K\frac{1+Z_{\gamma_1}}{1+Z_{\gamma_2}} &\quad\text{for $\gamma=\gamma_1$},\\[3pt]
                    \frac{K(1+Z_{\gamma_n})^2}{(1+Z_{\gamma_{n-1}})(1+Z_{\gamma_n'})} &\quad\text{for $\gamma=\gamma_n$},\\[3pt]
                    \frac{K(1+Z_{\gamma_p})^2}{(1+Z_{\gamma_{p-1}})(1+Z_{\gamma_{p+1}})} &\quad\text{for $\gamma=\gamma_p\neq\gamma_1,\gamma_n$}.
                \end{dcases}
            \end{displaymath}
        \end{proposition}
        
        \subsection{} In this subsection we assume that $\g=G(3)$. The roots of \g are expressed in terms of the elements $\varepsilon_1,\varepsilon_2,\varepsilon_3,\delta\in\h^*$, where $\varepsilon_1+\varepsilon_2+\varepsilon_3=0$, and the invariant form is determined by:
        \begin{gather*}
            (\varepsilon_1,\varepsilon_1)=(\varepsilon_2,\varepsilon_2)=-2(\varepsilon_1,\varepsilon_2)=-(\delta,\delta)=2 \quad\text{and}\quad (\varepsilon_i,\delta)=0.
            \intertext{The standard simple system of \g is given by}
            \Pi=\{\alpha_1=\varepsilon_1,\ \alpha_2=\varepsilon_2-\varepsilon_1,\ \alpha_3=\varepsilon_3+\delta\}.
            \intertext{We also have that}
            \Pi_0=\{\alpha_1,\alpha_2\}\text{\quad and\quad}\Phi_1^+=\{\delta,\ \pm\varepsilon_i+\delta\}.
        \end{gather*}
        All positive odd roots are isotropic with the exception of $\delta$. We denote the simple reflection corresponding to $\alpha_i$ by $s_i$ for $i=1,2$.

        Let $\lambda \neq \lambda_1, \lambda_2$ be a dominant integral singly atypical weight of atypicality type $\beta$. From ~\cite{MR3253284}, we see that the character of $V(\lambda)$ is again given by Equation \eqref{char} with the symbols in $G(3)$. Since in this case $\Pi_0$ has only two elements which are not mutually orthogonal, the coefficient of $X^\lambda$ in $-\log U(\lambda)$ is same as that in $\sfrac{\xi_1^2}{2}$, (cf. Equation \ref{xi 1 and xi 2}). For $\lambda=\lambda_1,\lambda_2$, the analogues of Equations \eqref{eq 1-xi in atyp} and \eqref{coef of xi k} are respectively:
        \begin{gather}
            U(\lambda)=\frac{\phantom{2}2+Z_\beta}{2(1+Z_\beta)}\left(1+\underbrace{\sum_{1\neq w\in\W}(-1)^{\lw}\frac{1+Z_\beta}{2+Z\beta}\cdot\frac{2+Z_{w\beta}}{1+Z_{w\beta}}X(\lambda,w)}_{-\xi}\right)=\frac{\phantom{2}2+Z_\beta}{2(1+Z_\beta)}(1-\xi) \tag{\ref*{eq 1-xi in atyp}a}.\label{eq 1-xi lam1 lam2}\\
            \xi_1^k= (-1)^k\mathlarger{\sum_{w_i\in\mathcal{I}}}(-1)^{\sum\ell(w_i)} \left(\frac{1+Z_\beta}{2+Z_\beta}\right)^k \mathlarger{\prod_{i=1}^k}\left(\frac{2+Z_{w_i\beta}}{1+Z_{w_i\beta}}\right) X(\lambda,w_i). \tag{\ref*{coef of xi k}a}\label{coef xi k lam1 lam2}
        \end{gather}
        Now from Equations \eqref{coef of xi k} and \eqref{coef xi k lam1 lam2}, we obtain the following:
        \begin{proposition}\label{coef for G3}
            With notations as above, the coefficient of $X^\lambda$ in $-\log U(\lambda)$ is given by
            \begin{displaymath}
                \begin{dcases}
                    \phantom{M^2}\frac{(1+Z_\beta)^2}{(1+Z_{s_1\beta})(1+Z_{s_2\beta})} &\quad\text{for $\lambda\neq \lambda_1,\lambda_2$},\\[5pt]
                    M^2\frac{(2+Z_{s_1\beta})(2+Z_{s_2\beta})}{(1+Z_{s_1\beta})(1+Z_{s_2\beta})} &\quad\text{for $\lambda= \lambda_1,\lambda_2$}.
                \end{dcases}
            \end{displaymath}
            Here $M=\frac{1+Z_\beta}{2+Z_\beta}$.
        \end{proposition}
        
        \subsection{} We now consider $\g=F(4)$. The invariant form on $\h^*$ is determined by
        \begin{gather*}
            (\varepsilon_i,\varepsilon_j)=\delta_{ij},\quad (\delta,\delta)=-3,\quad\text{and\quad} (\varepsilon_i,\delta)=0 \quad\text{for all } i,j\in\{1,2,3\}.
            \intertext{The standard simple system is given by}
            \Pi=\{\alpha_1=\varepsilon_1-\varepsilon_2,\ \alpha_2=\varepsilon_2-\varepsilon_3,\ \alpha_3=\varepsilon_3,\ \alpha_4=\tfrac{1}{2}(-\varepsilon_1-\varepsilon_2-\varepsilon_3+\delta)\}.
            \intertext{We also have that}
            \Pi_0=\{\alpha_1,\alpha_2,\alpha_3\}\quad\text{and}\quad \Phi_1^+=\{\tfrac{1}{2}(\pm\varepsilon_1\pm\varepsilon_2\pm\varepsilon_3+\delta)\}.
        \end{gather*}

        Here all the odd roots are isotropic. As before the simple reflection corresponding to $\alpha_i$ is denoted by $s_i$ for $i=1,2,3$. Consider now a dominant integral singly atypical weight $\lambda\neq\lambda_1,\lambda_2$ of \g with atypicality type $\beta$. Observe that no simple reflection fixes $\beta$. The character formula for the \g-module $V(\lambda)$ is same as that for $G(3)$. Here $\alpha_1$ and $\alpha_3$ are mutually orthogonal. In this case, we need to compute the coefficient of $X^\lambda$ in $\sfrac{\xi_1^2}{2}+\sfrac{\xi_1^3}{3}$. The coefficient in $\sfrac{\xi_1^3}{3}$ corresponds to the 3-partitions of $\Pi_0$, and it is equal to $\frac{2(1+Z_\beta)^3}{(1+Z_{s_1\beta})(1+Z_{s_2\beta})(1+Z_{s_3\beta})}$. There is only one 2-partition of $\Pi_0$, namely $(I_1,I_2)$ with $I_1=\{\alpha_1,\alpha_3\}$ and $I_2=\{\alpha_2\}$. This contributes to the coefficient of $X^\lambda$ in $\sfrac{\xi_1^2}{2}$ which is equal to $\frac{-(1+Z_\beta)^2}{(1+Z_{s_1s_3\beta})(1+Z_{s_2\beta})}$. The case when $\lambda=\lambda_1,\lambda_2$ goes over verbatim with the exception that Equation  \eqref{coef xi k lam1 lam2} is used instead of Equation \eqref{coef of xi k} to compute the coefficient of $\sfrac{\xi_1^k}{k}$. Consequently, We have the following:
        \begin{proposition}\label{coef for f4}
            With the notations as above, the coefficient of $X^\lambda$ in $-\log U(\lambda)$ is given by
            \begin{displaymath}
                \begin{dcases}
                    \phantom{2M^3}\frac{2(1+Z_\beta)^3}{(1+Z_{s_1\beta})(1+Z_{s_2\beta})(1+Z_{s_3\beta})}-\frac{(1+Z_\beta)^2}{(1+Z_{s_1s_3\beta})(1+Z_{s_2\beta})} &\quad\text{for $\lambda\neq\lambda_1,\lambda_2$},\\[5pt]
                    2M^3\frac{(2+Z_{s_1\beta})(2+Z_{s_2\beta})(2+Z_{s_3\beta})}{(1+Z_{s_1\beta})(1+Z_{s_2\beta})(1+Z_{s_3\beta})} - M^2\frac{(2+Z_{s_1s_3\beta})(2+Z_{s_2\beta})}{(1+Z_{s_1s_3\beta})(1+Z_{s_2\beta})} &\quad\text{for $\lambda=\lambda_1,\lambda_2$}.
                \end{dcases}
            \end{displaymath}
            Here $M=\frac{1+Z_\beta}{2+Z_\beta}$.
        \end{proposition}

        \subsection{} We now treat the Lie superalgebra $\g=\mathfrak{sl}(m+1,n+1)$ for $n>0$. In this case we have $\Pi_0=S\sqcup T$, with $S=\{\alpha_1,\dots,\alpha_m\}$ and $T=\{\beta_1,\dots,\beta_n\}$, where $\alpha_i=\varepsilon_i-\varepsilon_{i+1}$ and $\beta_j=\delta_j-\delta_{j+1}$. The simple reflection corresponding to $\alpha_i$ is denoted by $s_i$ and that corresponding to $\beta_j$ by $t_j$. The sets $S$ and $T$ are mutually orthogonal. We have that $\Phi_1^+=\{\gamma_{ij}\coloneqq\varepsilon_i-\delta_j\colon 1\leqslant i\leqslant m,\ 1\leqslant j\leqslant n\}$, and all of these roots are isotropic.

        Consider now $\lambda\in\h^*$, a dominant integral singly atypical weight of \g of type $\gamma_{pq}$. Note that $\gamma_{pq}$ is fixed by all simple reflections except $s_{p-1},s_p,t_{q-1}$ and $t_q$. Our aim is to calculate the coefficient of $X^\lambda$ in the expansion of $-\log U(\lambda)$. Since $(\alpha_i,\beta_j)=0$, the contribution of a $k$-partition to the coefficient of $X^\lambda$ is contingent upon how $\alpha_{p-1},\alpha_p,\beta_{q-1}$ and $\beta_q$ appear in that  particular partition. Therefore, all possible combinations of $k$-partitions have to be treated separately for any $k>1$. We put $B_{p,q}\coloneqq\{\alpha_{p-1},\alpha_p,\beta_{q-1},\beta_q\}$. Below we list down all distinct combinations of how $\alpha_{p-1},\alpha_p,\beta_{q-1}$ and $\beta_q$ can occur in a $k$-partition of $\Pi_0$.

        \begin{table}[ht]
                \begin{tabular}{|c|c|c|}\hline
                     $I_1$ & $I_2$ & \scriptsize{Contribution to the coeff. of $X^\lambda$ in $\sfrac{\xi_1^k}{k}$}\\ \hline
                     $\alpha_p,\beta_q$ & $\alpha_{p-1},\beta_{q-1}$ & $f_1=\frac{(1+Z_{\gamma_{pq}})^2}{(1+Z_{\gamma_{p+1,q+1}})(1+Z_{ \gamma_{p-1,q-1}})}$\\[3.5pt] \hline
                     $\alpha_p,\beta_{q-1}$ & $\alpha_{p-1},\beta_q$ & $f_2=\frac{(1+Z_{\gamma_{pq}})^2}{(1+Z_{\gamma_{p+1,q-1}})(1+Z_{\gamma_{p-1,q+1}})}$\\[3.5pt] \hline
                \end{tabular} 
                \caption{\textsc{2 parts}}{\label{2parts}}
        \end{table}
        In \autoref{2parts}, we describe all possible $k$-partitions $(I_1,\dots,I_k)$ where all the elements of $B_{p,q}$ occur as a pair. Note that, the number of occurrences of the two types of $k$-partitions in \autoref*{2parts} are same and we denote this number by $r_k^{(2)}$. Then from Equation \eqref{coef of xi k}, it follows that the aggregate contribution of all these partitions to the coefficient of $X^\lambda$ in $\sfrac{\xi_1^k}{k}$ is $(-1)^{m+n}\frac{(-1)^k}{k}r_k^{(2)}(f_1+f_2)$.

        In \autoref{3parts}, we consider all $k$-partitions where two elements of $B_{p,q}$ occur as a pair and the rest occur separately. 
        \begin{table}[ht]
                \begin{tabular}{|c|c|c|c|}\hline
                   $I_1$ & $I_2$ & $I_3$ & \scriptsize{Contribution to the coeff. of $X^\lambda$ in $\sfrac{\xi_1^k}{k}$}\\ \hline
                   $\alpha_{p-1},\beta_{q-1}$ & $\alpha_p$ & $\beta_q$ & $g_1=\frac{(1+Z_{\gamma_{pq}})^3}{(1+Z_{ \gamma_{p-1,q-1}}) (1+Z_{\gamma_{p+1,q}}) (1+Z_{\gamma_{p,q+1}})}$\\[3.5pt] \hline
                   $\alpha_{p-1},\beta_{q}$ & $\alpha_{p}$ & $\beta_{q-1}$ & $g_2=\frac{(1+Z_{\gamma_{pq}})^3}{(1+Z_{\gamma_{p-1,q+1}}) (1+Z_{\gamma_{p+1,q}}) (1+Z_{\gamma_{p,q-1}})}$\\[3.5pt] \hline
                   $\alpha_{p},\beta_{q-1}$ & $\alpha_{p-1}$ & $\beta_{q}$ & $g_3=\frac{(1+Z_{\gamma_{pq}})^3}{(1+Z_{\gamma_{p+1,q-1}}) (1+Z_{\gamma_{p-1,q}}) (1+Z_{\gamma_{p,q+1}})}$\\[3.5pt] \hline
                   $\alpha_{p},\beta_{q}$ & $\alpha_{p-1}$ & $\beta_{q-1}$ & $g_4=\frac{(1+Z_{\gamma_{pq}})^3}{(1+Z_{ \gamma_{p+1,q+1}}) (1+Z_{\gamma_{p-1,q}}) (1+Z_{\gamma_{p,q-1}})}$\\[3.5pt] \hline
                \end{tabular}
                \caption{\textsc{3 parts}}\label{3parts}
        \end{table}
        As before, we denote the common number of all such $k$-partitions by $r_k^{(3)}$. Then the aggregate contribution of all these partitions to the coefficient of $X^\lambda$ is $(-1)^{m+n}\frac{(-1)^k}{k}r_k^{(3)}(g_1+g_2+g_3+g_4)$.

        \begin{table}[ht]
            \begin{tabular}{|c|c|c|c|c|}\hline
                 $I_1$ & $I_2$ & $I_3$ & $I_4$ & \scriptsize{Contribution to the coeff. of $X^\lambda$ in $\sfrac{\xi_1^k}{k}$}\\ \hline
                 $\alpha_{p-1}$ & $\beta_{q-1}$ & $\alpha_p$ & $\beta_q$ & $h_1=\frac{(1+Z_{\gamma_{pq}})^4}{(1+Z_{\gamma_{p-1,q}}) (1+Z_{\gamma_{p,q-1}}) (1+Z_{\gamma_{p+1,q}})(1+Z_{\gamma_{p,q+1}})}$\\[3.5pt] \hline
            \end{tabular}
            \caption{\textsc{4 parts}}\label{4parts}
        \end{table}
         Finally, in \autoref{4parts}, we describe all $k$-partitions of $\Pi_0$ where each element of $B_{p,q}$ occurs separately. Let $r_k^{(4)}$ denote the total number of all such partitions. Then the aggregate contribution of all the $k$-partitions depicted in \autoref*{4parts} to the coefficient of $X^\lambda$ is given by $(-1)^{m+n}\frac{(-1)^k}{k}r_k^{(4)}h_1$. As both $(\alpha_{p-1},\alpha_p)$ and $(\beta_{q-1},\beta_q)$ are nonzero, by definition of a $k$-partition we conclude that \cref{2parts,3parts,4parts} exhaust all possible combinations.
         
         We denote by $A_k$ the coefficient of $X^\lambda$ in $\sfrac{\xi_1^k}{k}$. It now follows that
         \begin{gather}\label{coef of f1 in Ak}
             A_k=(-1)^{m+n}\cdot \frac{(-1)^k}{k}[r_k^{(2)}(f_1+f_2)+r_k^{(3)}(g_1+g_2+g_3+g_4)+r_k^{(4)}h_1].
         \end{gather}
         With the notations just introduced, we find that the coefficient of $X^\lambda$ in $-\log U(\lambda)$ is given by $A\coloneqq A_2+\dots+A_{m+n}$, and it is a $\mathbb C$-linear combination of $f_1,f_2,g_1,g_2,g_3,g_4,h_1$. It remains to prove that $A\neq 0$, which is contention of the next proposition.
         
         \begin{proposition}\label{coef of sl mn}
             With the notations as above, the coefficient of $X^\lambda$ in $-\log U(\lambda)$ is nonzero, i.e. $A\neq 0$.
         \end{proposition}
         \begin{proof}
             It is not difficult to see that the set $\{f_1,f_2,g_1,g_2,g_3,g_4,h_1\}$ is linearly independent over $\mathbb C$, therefore it suffices to show that the coefficient of $f_1$ in $A$ is nonzero. From Equation \eqref{coef of f1 in Ak}, we find that this coefficient is $(-1)^{m+n} \sum_{k\geqslant 2}  \frac{(-1)^k}{k} r_k^{(2)}$, where the sum runs up to $k=m+n$. Recall that $P_k(\Pi_0)$ is the set of all $k$-partitions of $\Pi_0$. Let $Q_k$ be the subset of $P_k(\Pi_0)$ defined by
             \begin{displaymath}
                 Q_k\coloneqq\{\mathcal{Q}=(I_1,\dots,I_k)\in P_k(\Pi_0)\colon \alpha_{p-1},\beta_{q-1}\in I_{i_1}\text{ and } \alpha_p,\beta_q\in I_{i_2}\text{ for some } 1\leqslant i_1\neq i_2\leqslant k\}.
             \end{displaymath}
             By definition, $r_k^{(2)}$ is the cardinality of $Q_k$. Notice that $r_k^{(2)}=0$ if $k>m+n-2$. Since both $\{\alpha_{p-1},\beta_{q-1}\}$ and $\{\alpha_p,\beta_q\}$ occur as pairs in any $\mathcal{Q}\in Q_k$, they can be treated as a single symbol. Let $\nu_{p-1,q-1}$ (resp. $\nu_{p,q}$) stand for the pairs $\{\alpha_{p-1},\beta_{q-1}\}$ (resp. $\{\alpha_p,\beta_q\}$). Now we consider the following set:
             \begin{displaymath}
                 G_{p,q}\coloneqq \{\alpha_1,\dots,\alpha_{p-2},\hat\alpha_{p-1},\hat\alpha_p,\dots,\alpha_m,
                 \beta_1,\dots,\beta_{q-2},\hat\beta_{q-1},\hat\beta_q,\dots,\beta_n\}\sqcup \{\nu_{p-1,q-1},\nu_{p,q}\},
             \end{displaymath}
             where hat denotes omission. Then $|G_{p,q}|=m+n-2$. We intend to construct a graph $\mathcal{G}_{p,q}$ having $G_{p,q}$ as the set of vertices. The vertices in $G_{p,q}\backslash \{\nu_{p-1,q-1},\nu_{p,q}\}$ being a subset of $\Pi_0$, are joined in accordance with the usual rule (\textit{see} \vref{graph}); $\alpha_i$ is joined to $\nu_{p-1,q-1}$ (resp. to $\nu_{p,q}$) by an edge if $(\alpha_i,\alpha_{p-1})\neq 0$ (resp. $(\alpha_i,\alpha_p)\neq 0$). Similarly, $\beta_j$ is joined to $\nu_{p-1,q-1}$ (resp. to $\nu_{p,q}$) by an edge if $(\beta_j,\beta_{q-1})\neq 0$ (resp. $(\beta_j,\beta_q)\neq 0$). As both $(\alpha_{p-1},\alpha_p)$ and $(\beta_{q-1},\beta_q)$ are nonzero, $\nu_{p-1,q-1}$ and $\nu_{p,q}$ is connected by an edge. These rules define the graph $\mathcal{G}_{p,q}$. We see that $\nu_{p-1,q-1}$ is joined to both $\alpha_{p-2}$ and $\beta_{q-2}$, while $\nu_{p,q}$ is connected to $\nu_{p-1,q-1}$, $\beta_{q+1}$ and $\alpha_{p+1}$. This means that $\mathcal{G}_{p,q}$ is a \emph{tree}.

             Evidently, $r_k^{(2)}=P_k(\mathcal{G}_{p,q})$ for $2\leqslant k\leqslant m+n-2$, and $P_1(\mathcal{G}_{p,q})=0$. This implies that
             \begin{displaymath}
                 (-1)^{m+n-2} \sum\nolimits_{k=2}^{m+n-2} \frac{(-1)^k}{k} r_k^{(2)}=(-1)^{|G_{p,q}|}\sum\nolimits_{k=2}^{|G_{p,q}|} \frac{(-1)^k}{k} P_k(\mathcal{G}_{p,q})=k(\mathcal{G}_{p,q}).
             \end{displaymath}
             Since $\mathcal{G}_{p,q}$ is a tree, by ~\cite{MR2980495}*{Corollary 1} we conclude that $k(\mathcal{G}_{p,q})=1$.
         \end{proof}
         
        Summarizing \cref{coef for sl n 1,coef for cn,coef for G3,coef for f4,coef of sl mn}, we obtain the following important Lemma.
        \begin{lemma}\label{nonzero lem atyp}
            Let \g be any one among the Lie superalgebras mentioned in \autoref{ x lam to u lam atyp}. Let $\lambda\in\h^*$ be dominant integral singly atypical of type say $\gamma$. Then the coefficient of $X^\lambda$ in the expansion of $-\log U(\lambda)$ is nonzero.
        \end{lemma}

        The next theorem will be used in the proof of the main result of this section (cf. \autoref{permfac}).
        
        \begin{theorem}\label{thm: perm of ulmbda atyp}
            Let \g be as in \autoref{ x lam to u lam atyp}. Let $\nu_1,\dots,\nu_r;\mu_1,\dots,\mu_s$ be dominant integral singly atypical weights of \g of atypicality type $\gamma$.  Suppose we are given the equality
            \begin{gather}\label{eq of pdt of u lam an u mu}
                U(\nu_1)\cdots U(\nu_r)=U(\mu_1)\cdots U(\mu_s).
            \end{gather}
            Then $r=s$, and there is a permutation \(\sigma\) of \(\{1,\dots,r\}\) such that $U(\nu_k)=U(\mu_{\sigma(k)})$.
        \end{theorem}
        \begin{proof}
             First we assume that none of these weights are either $\lambda_1$ or $\lambda_2$ in case when $\g=G(3),F(4)$. As in Equation \eqref{eq 1-xi in atyp}, we express $U(\nu_i)$ in the form $U(\nu_i)=\frac{1}{1+Z_{\gamma}}(1-\xi_i)$, and similarly $U(\mu_j)=\frac{1}{1+Z_{\gamma}}(1-\zeta_j)$. Taking negative logarithm on both sides of Equation \eqref{eq of pdt of u lam an u mu} yields:
            \begin{gather}
                \sum_{i=1}^r -\log (1-\xi_i)= \sum_{j=1}^s -\log (1-\zeta_j) +(s-r)\log (1+Z_{\gamma}).\label{-log of pdt}
                \intertext{After applying $\Theta_{\Pi_0}$ (\textit{see} \vpageref{theta c}) on both sides we obtain (cf. \autoref{lowestdeg})}
                \sum_{i=1}^r (L\cdot X^{\nu_i} +\vartheta_i)= \sum_{j=1}^s (L\cdot X^{\mu_j} +\varphi_j),\notag
            \end{gather}
            where $\vartheta_i$ (resp. $\varphi_j$) is the sum of monomials of total degree bigger than $\deg X^{\nu_i}$ (resp. $X^{\mu_j}$), and $L$ is the coefficient of $X^{\nu_i}$ and $X^{\mu_j}$ for all $1\leqslant i\leqslant r,\ 1\leqslant j\leqslant s$. Since all the weights are of same atypicality type, $L$ does not depend on the weights, and it is nonzero by \autoref{nonzero lem atyp}.

            Among all the $X^{\nu_i}$'s, we choose one having lowest possible degree. Without any loss of generality, say this element is $X^{\nu_1}$; hence, this monomial must appear in the right hand side of the above equation. By minimality of degree, we must have $X^{\nu_1}=X^{\mu_k}$, for some $1\leqslant k\leqslant s$. It follows that $U(\nu_1)=U(\mu_k)$ by \autoref{ x lam to u lam atyp}.
            
            To show $r=s$, we assume on contrary that $r>s$. Proceeding via induction, we see that all the factors on the right side get canceled with some in the left side, forcing $r=s$, as a product of $U(\nu_i)$'s is not identically 1.

            Now for $\g=G(3)\text{ and }F(4)$, we assume that $\lambda_1$ and $\lambda_2$ appear collectively $m$ times on the left hand side, and $n$ times on the right hand side of Equation \eqref{eq of pdt of u lam an u mu}. In this case, Equation \eqref{-log of pdt} will take the following form (cf. Equation \ref{eq 1-xi lam1 lam2}):
            \begin{displaymath} \tag{\ref*{-log of pdt}a}
                \sum_{i=1}^r -\log (1-\xi_i)= \sum_{j=1}^s -\log (1-\zeta_j) +(s-r)\log (2+2Z_{\gamma}) +(m-n)\log(2+Z_\gamma).
            \end{displaymath}
            The rest of the proof now goes over verbatim.
        \end{proof}
        
        We are now ready to prove the main result of this section.
        
        \begin{theorem}\label{Thm:atyp}
            Let \g be any Lie superalgebra among $\mathfrak{sl}(m+1,n+1),\ \mathfrak{osp}(2,2n),\ G(3)$, or $F(4)$. Suppose we are given the following isomorphism of \g-modules \[V(\nu_1)\otimes\dots\otimes V(\nu_r)\cong V(\mu_1)\otimes\dots\otimes V(\mu_s),\] where each $\nu_i$'s and $\mu_j$'s are dominant integral singly atypical of type $\gamma$. Then $r=s$, and there is a permutation \(\sigma\) of \(\{1,\dots,r\}\) such that \(V(\nu_k)\cong V(\mu_{\sigma(k)})\) for all $k$.
        \end{theorem}
        \begin{proof}
            After taking characters on both sides of the above isomorphism and proceeding as in the proof of \autoref{Thm:tnsrpdt}, we get that
            \begin{displaymath}
                U(\nu_1)\cdots U(\nu_r)D^{r-s}=U(\mu_1)\cdots U(\mu_s),\quad\text{where $D=\tfrac{\prod_{\alpha\in\Phi^+_1}(1+e^{-\alpha})}{\prod_{\alpha\in\Phi^+_0}(1-e^{-\alpha})}$}.
            \end{displaymath}
            By \autoref{thm: perm of ulmbda atyp}, we conclude that $r=s$ and $U(\nu_k)=U(\mu_{\sigma(k)})$ for some permutation \(\sigma\) of \(\{1,\dots,r\}\) and $1\leqslant k\leqslant r$. This implies by \autoref{ x lam to u lam atyp} that $\nu_k=\mu_{\sigma(k)}$, whence it follows that $V(\nu_k)\cong V(\mu_{\sigma(k)})$.
        \end{proof}

	\bibliography{References} 
\end{document}